\documentclass[smallextended]{svjour3}
\usepackage{amssymb}
\usepackage{amsmath}
\usepackage{amsfonts}

\allowdisplaybreaks

\def\XXint#1#2#3{{\setbox0=\hbox{$#1{#2#3}{\int}$ }
\vcenter{\hbox{$#2#3$ }}\kern-.6\wd0}}

\DeclareMathOperator{\divergence}{div}

\numberwithin{equation}{section}

\title{ Reconstruction from  boundary measurements  for less regular conductivities}
\author{\text{${}^{a}$Andoni Garc\'ia  \and  ${}^{b}$Guo Zhang}  \thanks{ ${}^{b}$  Corresponding author\ Telephone: +358 40 8053453, Fax: +358 40 8053453}
} 
\institute{${}^{a}$Department of Mathematics and Statistics, P.O. Box 68 (Gustaf H\"{a}llstr\"{o}min katu 2b)
FI-00014
{\small{ University of Helsinki, Finland}} \and
{\small${}^{b}$ Department of Mathematics and Statistics, P.O. Box 35 (MaD), FI-40014} 
{\small University of Jyv\"{a}skyl\"{a}, Finland }
\email{${}^{a}$andoni.garcia@helsinki.fi \and   ${}^{b}$guozhang@jyu.fi }}

\begin{document}

\date{}
\maketitle
%

\begin{abstract}
In this paper, following Nachman's idea \cite{Na1} and Haberman and Tataru's idea \cite{HT1},  we reconstruct $C^1$ conductivity $\gamma$ 
or Lipchitz conductivity $\gamma$ with small enough value of $|\nabla log\gamma|$ in a Lipschitz domain $\Omega$ from the Dirichlet-to-Neumann
map $\Lambda_{\gamma}$. In the appendix the authors and R. M. Brown recover the gradient of a $C^1$-conductivity at the boundary of a Lipschitz domain from the Dirichlet-to-Neumann
map $\Lambda_{\gamma}$.
 \end{abstract}
 
\subclass{35R30}
\keywords{Inverse conductivity problem, Dirichlet-to-Neumann map, Calder\'{o}n problem, Boundary integral equation, Bourgain's space}

\tableofcontents
\setcounter{tocdepth}{1}

\section{Introduction}
Let $\Omega\subset\mathbb{R}^n, n\geq2$ be an open, bounded domain  and let $\gamma$ be a strictly positive real valued function
defined on $\Omega$ which gives the conductivity at a given point. Given a voltage potential $f$ on the boundary, the equation for the potential in the interior, under
the assumption of no sinks or sources of current in $\Omega$, is
\begin{equation}\label{0.1}
\renewcommand{\arraystretch}{1.25}
\begin{array}{lll}
\divergence(\gamma\nabla u)=0,\ \ \ \ \text{in}\ \Omega,\ \ \ \ u$\textbar$_{\partial\Omega}=f.
\end{array}
\end{equation}
The Dirichlet-to-Neumann map is defined in this case as follows:
\[
\renewcommand{\arraystretch}{1.25}
\begin{array}{lll}
\Lambda_{\gamma}(f)=\gamma\frac{\partial u}{\partial \nu}$\textbar$_{\partial\Omega},
\end{array}
\]
where $\frac{\partial}{\partial \nu}$ is the outward normal derivative at the boundary. For $\gamma\in L^{\infty}(\Omega)$ and being the boundary of $\Omega$ Lipschitz, then
 $\Lambda_{\gamma}$ is a well defined map
from $H^{\frac{1}{2}}(\partial\Omega)$ to $H^{-\frac{1}{2}}(\partial\Omega)$.

The Calder\'{o}n problem concerns the inversion of the map $\gamma\rightarrow\Lambda_{\gamma}$, i.e., whether $\Lambda_{\gamma}$ determines
 $\gamma$ uniquely and in that
case how to reconstruct $\gamma$ from $\Lambda_{\gamma}$.

For the boundary determination: if $\partial{\Omega}$ is $C^{\infty}$ and $\gamma\in C^{\infty}(\bar{\Omega})$, Kohn and Vogelius \cite{KV1}
showed that $\Lambda_{\gamma}$ determines $\frac{\partial^k\gamma}{\partial \nu^k}\arrowvert_{\partial{\Omega}}$ for all $k\geq 0$; if $\partial{\Omega}$ is
Lipschitz and $\gamma\in Lip(\Omega)$, Brown \cite {Bro1} showed that $\gamma\arrowvert_{\partial{\Omega}}$ can be recovered from the knowledge of $\Lambda_{\gamma}$; if $\partial{\Omega}$ is
$C^1$ and $\gamma\in C^{1}(\bar{\Omega})$, Nakamura and Tanuma \cite{NT1} showed that $\frac{\partial\gamma}{\partial \nu}\arrowvert_{\partial{\Omega}}$ can be 
recovered from $\Lambda_{\gamma}$. Now, being $\gamma\in C^{1}(\bar{\Omega})$ and for $\Omega$ Lipschitz domain, the authors and  Brown together show the recovering of the
 gradient of the conductivity at the boundary (see Appendix).

For the higher dimensional problem $(n\geq 3)$, uniqueness was first proven for piecewise analytic conductivities by Kohn and Vogelius \cite{KV2}. Later, Sylvester and 
Uhlmann showed that the uniqueness holds for smooth conductivities in their fundamental paper \cite{SU1} which opened the door for studying the Calder\'{o}n problem. 
Generalization to less regular conductivities had been obtained by several authors. Brown \cite{B1} obtained uniqueness under the assumption of 
$\frac{3}{2}+\epsilon$ derivatives. P\"{a}iv\"{a}rinta, Panchenko and Uhlmann \cite{PPU1} proved uniqueness under the assumption of $\frac{3}{2}$ bounded derivatives. 
Brown and Torres \cite{BT1} obtained uniqueness under the assumption of $\frac{3}{2}$ derivatives being in $L^p, p>2n$. Later, Uhlmann \cite{Uhl1} proposed a conjecture 
whether uniqueness holds in dimension $n\geq 3$ for Lipschitz or less regular conductivities.

Recently, Haberman and Tataru \cite{HT1} gave a partial answer to this conjecture. They used a totally new idea to construct CGO solutions in Bourgain's space and
showed uniqueness for $C^1$ conductivity $\gamma$ 
or Lipchitz conductivity $\gamma$ with small enough value of $|\nabla log\gamma|$. The ideas and techniques coming from this work have been widely used as can be seen in the papers by Zhang \cite{Z}, Caro, Garc\'ia and Reyes \cite{CGR} or Caro and Zhou \cite{CZ}.

If uniqueness holds, one whether or not  construct the conductivity in the domain $\Omega$ by the boundary measurements. Nachman \cite{Na1} and Novikov \cite{Nov1} independently 
solved  this problem for $C^2$ conductivity.

We will briefly describe Nachman's idea as follows. For $\gamma\in C^2(\bar{\Omega})$, if $u$ is a solution to \eqref{0.1}, Sylvester and Uhlmann reduced 
$v=\gamma^{\frac{1}{2} }u$ to satisfy 
\begin{equation}\label{0.2}
 -\triangle v+qv=0\ \ \ \ \ \ \text{in}\  \Omega,
\end{equation}
where $q=\frac{\triangle \gamma^{1/2}}{\gamma^{1/2}}$, and the corresponding Dirichlet-to-Neumann map $\Lambda_{q}$ to \eqref{0.2} is determined by $\Lambda_{\gamma}$.
For $\rho_{1}^{(n)} \cdot \rho_{1}^n=0,  \rho_{1}^{(n)} \in  \mathbb{C}^n $ and $|\rho_{1}^{(n)}|\rightarrow \infty,\  \text{as}\ n\rightarrow\infty$, Sylvester and Uhlmann constructed a 
family of CGO solutions $v^{(n)}=e^{x \cdot \rho_{1}^{(n)}}(1+\Phi_1^{(n)}(x))$  to
\begin{equation}\label{0.3}
 -\triangle v+qv=0\ \ \ \ \ \ \text{in}\  \mathbb{R}^n,
\end{equation}
with the remainder term $\Phi_1^{(n)}(x)$ decaying to zero in some sense as $|\rho_{1}^{(n)}|\rightarrow \infty$,
where $q=\frac{\triangle \gamma^{1/2}}{\gamma^{1/2}}$ in $\Omega$ and $q=0$ outside $\Omega$.

For the appropriate chosen $\rho_{2}^{(n)} \cdot \rho_{2}^n=0,  \rho_{2}^{(n)} \in  \mathbb{C}^n $ and $\rho_{1}^{(n)}+\rho_{2}^{(n)}=ik$, Green's formula gives us 
\begin{equation}\label{0.4}
 \displaystyle\int_{\mathbb{R}^n} q(x)e^{ix \cdot k}(1+\Phi_1^{(n)}(x))dx =
\displaystyle\int_{\partial{\Omega}}(\Lambda_{q}e^{x \cdot \rho_{2}^{(n)}}- \frac{\partial{e^{x \cdot \rho_{2}^{(n)}}}}{\partial{\nu}})v^{(n)}dS,
\end{equation}
where 
\begin{equation}
\Lambda_q=\left(\gamma^{-1/2}\Lambda_\gamma\gamma^{-1/2}+\frac{1}{2}\gamma^{-1}\frac{\partial\gamma}{\partial\nu}\right)|_{\partial\Omega}.
\end{equation}
Letting $n\rightarrow \infty$ in \eqref{0.4} we conclude from the decay property of $\Phi_1^{(n)}(x)$
\begin{equation}\
 \hat{q}(k)=
\lim\limits_{n\rightarrow \infty}\displaystyle\int_{\partial{\Omega}}(\Lambda_{q}e^{x \cdot \rho_{2}^{(n)}}-
 \frac{\partial{e^{x \cdot \rho_{2}^{(n)}}}}{\partial{\nu}})v^{(n)}dS,
\end{equation}
so the problem is then to recover the boundary values $v^{(n)}\arrowvert_{\partial{\Omega}}=f_{\rho_{1}^{(n)}}$.

On the othe hand, $v^{(n)}$ is a solution to the exterior problem:
\begin{equation}\label{eq:0.5}
 \left\{
\renewcommand{\arraystretch}{1.25}
\begin{array}{lll}
-\triangle {v^{(n)} }=0 \ \ \ \ \ \  \ \ \ \ \text{in}\  \mathbb{R}^n\backslash \bar{\Omega},\\
v^{(n)}\arrowvert_{\partial{\Omega}}=f_{\rho_{1}^{(n)}},\ \ \ \ \frac{\partial v^{(n)}}{\partial{\nu}}\arrowvert_{\partial{\Omega}}=\Lambda_{q}f_{\rho_{1}^{(n)}}.
\end{array}
\right.
\end{equation}
If $v^{(n)}$ satisfies the radiation condition 
\begin{equation}\label{0.6}
 \lim\limits_{R\rightarrow \infty}\displaystyle\int_{|y|=R}G_{\rho_{1}^{(n)}}(x, y)\frac{\partial(v^{(n)}-e^{x \cdot \rho_{1}^{(n)}})}{\partial{\nu}(y)}
-\frac{\partial{G_{\rho_{1}^{(n)}}(x, y)}}{\partial{\nu}(y)}(v^{(n)}-e^{x \cdot \rho_{1}^{(n)}})dS(y)=0,
\end{equation}
using Green's formula in \eqref{eq:0.5} we can reduce $f_{\rho_{1}^{(n)}}$  to satisfy a boundary integral equation 
\begin{equation}\label{0.7}
f_{\rho_{1}^{(n)}}=e^{x \cdot \rho_{1}^{(n)}}-(S_{\rho_{1}^{(n)}}\Lambda_{q}-B_{\rho_{1}^{(n)}}-\frac{1}{2}I)f_{\rho_{1}^{(n)}} 
\end{equation}
and \eqref{0.7} is uniquely solvable by Fredholm alternative theory, where $G_{\rho_{1}^{(n)}}, S_{\rho_{1}^{(n)}} \ \text{and}\ B_{\rho_{1}^{(n)}} $ are
defined in Section $3$.

Unfortunately, it is not easy to directly check that  $v^{(n)}$ satisfy \eqref{0.6}. To move around  this obstacle, Nachman \cite{Na1} dealt with this problem in an inverse direction:
He first constructed CGO solutions to \eqref{0.3} automatically satisfying the radial condition \eqref{0.6}  from boundary integral equation \eqref{0.7} and 
then showed that these solutions are the same as the solutions obtained by Sylvester and Uhlmann\cite{SU1}.

For $\gamma \in C^1(\bar{\Omega})$ 
or $\gamma \in Lip(\Omega)$ with the small enough value of $|\nabla log\gamma|$ we will follow Nachman's idea to construct CGO solutions to the equation 
\begin{equation}\label{0.8}
 \divergence(\gamma\nabla u)=0,\ \ \ \ \text{in}\ \mathbb{R}^n
\end{equation}
from the boundary integral equation \eqref{0.7}. In view of the less smooth regular $\gamma$, we need to do some changes in the above steps. First we reduce conductivity
$\gamma$ to the case $\gamma\equiv 1$ near the boundary of $B_{R}(0)\supset \Omega$ and show that the new Dirichlet-to-Neumann map $\tilde{\Lambda}_{\gamma}$ 
corresponding to $\gamma$ in $B_{R}(0)$ is determined by $\Lambda_{\gamma}$. Next we construct CGO solutions to \eqref{0.8} from the boundary integral
 equation \eqref{0.7} on the boundary $\partial{B}_{R}(0)$ and show that these solutions are the same as the solutions obtained by Haberman and Tataru \cite{HT1}.

We state the theorem as follows.
\begin{theorem}\label{t0.1}
 The conductivity $\gamma$ can be recovered in the domain $\Omega$ by the knowledge of  $\Lambda_{\gamma}$ under one of the following assumptions:\\
(a)\ Let $\Omega \subset \mathbb{R}^n, n\geq 3$ be a bounded domain with  Lipschitz boundary and let $\gamma(x) \in C^1(\bar{\Omega})$ be a real valued 
function with $\gamma(x) \geq C_0>0.$\\
(b)\  Let $\Omega \subset \mathbb{R}^n, n\geq 3$ be a bounded domain with Lipschitz boundary and let $\gamma(x) \in Lip({\Omega})$ be a real valued 
function with $\gamma(x) \geq C_0>0$ and such that there exists a constant $\varepsilon_{{n, \Omega}}$ satisfying $|\nabla {log\gamma}(x)| < \varepsilon_{{n, \Omega}}$.
\end{theorem}
Our paper is organized as follows. In Section $2$ we reduce the conductivity $\gamma$ to be $\gamma\equiv 1$ near the boundary. In Section $3$ we construct CGO solutions
from the boundary integral equation. In Section $4$ we state the reconstruction of the conductivity $\gamma$ from $\Lambda_{\gamma}$. In the appendix the authors and 
R. M. Brown give the proof of the recovering of the gradient of a $C^1$-conductivity at the boundary of a Lipschitz domain.

\section{Reduction to The Case $\gamma \equiv 1$ near The Boundary}

For a bounded domain $\Omega$ with Lipschitz boundary and $\gamma \in Lip (\Omega)$, if $\Lambda_{\gamma}$ is known, we can recover the values  of $\gamma $
at the boundary $\partial{\Omega}$ (see \cite{Bro1} ). From this knowledge we can extend $\gamma $ to be 
a Lipchitz function  in $\mathbb{R}^n$ with $\gamma \equiv 1$  outside a large ball $B_{R_0}$ and 
$\gamma(x)\geq C_0 > 0$. Furthermore, If $|\nabla {log\gamma}(x)| < \varepsilon_{{n, \Omega}}$ we can keep this property holding in $\mathbb{R}^n$.
For a bounded Lipschitz domain $\Omega$  and $\gamma \in C^1 (\bar {\Omega})$,  if $\Lambda_{\gamma}$ is known, 
we can recover the values  of $\gamma $ and the gradient of $\gamma $ at the boundary $\partial{\Omega}$ (see the appendix). Next, 
Whitney's extension allows us from the information of $ \partial^\alpha \gamma(x) $ for
 all $ x \in \partial \Omega $ and all $ |\alpha| \leq 1 $  to
extend $\gamma $ to be $C^1$ in $\mathbb{R}^n$ with $\gamma \equiv 1$  outside a large ball $B_{R_0}$ and 
$\gamma(x)\geq C_0 > 0$(The readers are referred to see Chapter VI \S4.7 of \cite{St}). In the above two cases obviously $\gamma$ is known in $\mathbb{R}^n \backslash \Omega$.

Now for fixed $R>R_{0}$ such that $\Omega\subset B_{R}(0)$, we define the new Dirichlet-to-Neumann map as follows:
\[
\renewcommand{\arraystretch}{1.25}
\begin{array}{lll}
 \tilde{\Lambda}_{\gamma}:\ \ \ \  f$\textbar$_{\partial{B}_R(0)}\longrightarrow  \frac{\partial{u}_{f}}{\partial{\nu}}$\textbar$_{\partial{B}_R(0)},
\end{array}
\]
where $u_{f}$ is a solution to 
\[
\left\{
\renewcommand{\arraystretch}{1.25}
\begin{array}{lll}
\divergence(\gamma\nabla u)=0,\ \ \ \ \text{in}\ B_{R}(0),\\
 u$\textbar$_{\partial{B}_R(0)}=f.
\end{array}
\right.
\]

In the following, we will show $\tilde{\Lambda}_{\gamma}$ is determined by $\Lambda_{\gamma}$ and the value of $\gamma$ in $\bar{B}_{R}(0)\backslash \Omega$. Here 
we need to mention that Nachman already used this idea in \cite{Na2} and obtained an exact formula in the plane with the the conductivity $\gamma\in W^{2,p}, p>1$ (see
Proposition $6.1$ of \cite{Na2}). Following Nachman's idea, we will generalize this formula in a bounded Lipchitz domain $\Omega \subset \mathbb{R}^n, n\geq2, $ for a Lipchitz 
conductivity $\gamma$.

In fact, if we denote by $\Omega_{1}$, respectively $\Omega_{2}$, the domains $\Omega$ and $B_{R}(0)$, and by $\partial{\Omega_{1}}$, respectively
 $\partial{\Omega_{2}}$, the boundary of $\Omega$ and the boundary of $B_{R}(0)$, the Dirichlet-to-Neumann map corresponding to $\gamma$ in the domain 
$\Omega_{2}\backslash \bar{\Omega}_1$ can be viewed as $2\times2$ matrix of operators $\Lambda^{ij}, 
H^1({\partial{\Omega}_{i}})\rightarrow  H^{\frac{1}{2}}({\partial{\Omega}_{j}}), i,j=1,2.,$ defined as follows. Given $f_{j}\in H^1({\partial{\Omega}_{j}})$ 
for $j=1,2.,$ considering the Dirichlet problem
\[
\left\{
\renewcommand{\arraystretch}{1.25}
\begin{array}{lll}
\divergence(\gamma\nabla w)=0,\ \ \ \ \text{in}\ \Omega_{2}\backslash \bar{\Omega}_1 ,\\
 w$\textbar$_{\partial{\Omega}_1}=f_{1},\ \ w$\textbar$_{\partial{\Omega}_2}=f_{2},
\end{array}
\right.
\]
we have the following Dirichlet-to Neumann map:$H^1({\partial{\Omega}_{j}})\rightarrow  L^{2}({\partial{\Omega}_{j}})$ as follows:
\begin{equation}\label{1.1}
\begin{pmatrix} 
-\gamma \frac{\partial{w}}{\partial{\nu_{+}}}\arrowvert_{\partial{\Omega}_{1}}\\
\gamma \frac{\partial{w}}{\partial{\nu_{-}}}\arrowvert_{\partial{\Omega}_{2}}
\end{pmatrix} 
=
\begin{pmatrix}
\Lambda^{11} & \Lambda^{12}\\ 
\Lambda^{21} & \Lambda^{22}
 \end{pmatrix}
\begin{pmatrix}
 f_{1}\\
f_{2}
\end{pmatrix},
\end{equation}
where $ \frac{\partial{w}}{\partial{\nu_{+}}}\arrowvert_{\partial{\Omega}_{1}}$ denotes the
 nontangential limit of $\frac{\partial{w}}{\partial{\nu}}$ from outside of $\bar{\Omega}_{1}$ 
and $\frac{\partial{w}}{\partial{\nu_{-}}}\arrowvert_{\partial{\Omega}_{2}}$ denotes the nontangential 
limit of $\frac{\partial{w}}{\partial{\nu}}$ from inside of ${\Omega}_{2}$. Since $\gamma$ is known in $\bar\Omega_{2}\backslash {\Omega}_1$, 
$\Lambda^{ij}, i, j=1,2.,$ are determined. Now for any $f\in H^{1}(\partial{\Omega}_{2})$ we have a solution $u\in H^{\frac{3}{2}}(\Omega_{2})
\bigcap C^s(\Omega_{2})$, where $0\leq s <2$, satisfying 
\[
\left\{
\renewcommand{\arraystretch}{1.25}
\begin{array}{lll}
\divergence(\gamma\nabla u)=0,\ \ \ \ \text{in}\ \Omega_{2},\\
 u$\textbar$_{\partial{\Omega}_2}=f.
\end{array}
\right.
\]
Let $g=u\arrowvert_{\partial{\Omega}_1}\in H^1(\partial{\Omega}_1)$. We have that \eqref{1.1} implies 
\begin{equation}\label{1.2}
 \tilde{\Lambda}_{\gamma}=\gamma \frac{\partial{w}}{\partial{\nu_{-}}}\arrowvert_{\partial{\Omega}_{2}}=\Lambda^{21}g+\Lambda^{22}f.
\end{equation}
If $g$ can be recovered from $\Lambda^{ij},  \Lambda_{\gamma} \  \text{and} \  f$, then $\tilde{\Lambda}_{\gamma}$ is known. On the other hand, in 
view of $u\in H^{\frac{3}{2}}(\Omega_{2})
\bigcap C^s(\Omega_{2}),$ we can deduce from \eqref{1.1} 
\[
 \Lambda^{11}g+\Lambda^{12}f=-\gamma \frac{\partial{u}}{\partial{\nu_{+}}}\arrowvert_{\partial{\Omega}_{1}}=-\Lambda_{\gamma} g.
\]
Hence, it follows that 
\[
 (-\Lambda_{\gamma}-\Lambda^{11})g=\Lambda^{12}f.
\]
If we can show that $(-\Lambda_{\gamma}-\Lambda^{11})$ is an invertible operator: $H^1({\partial{\Omega}_{1}})\rightarrow  L^{2}({\partial{\Omega}_{1}})$, 
then we have 
\begin{equation}\label{1.3}
 g=(-\Lambda_{\gamma}-\Lambda^{11})^{-1}\Lambda^{12}f.
\end{equation}
Finally \eqref{1.2} and \eqref{1.3} imply
\begin{equation}\label{1.4}
  \tilde{\Lambda}_{\gamma}=\Lambda^{21}(-\Lambda_{\gamma}-\Lambda^{11})^{-1}\Lambda^{12}f+\Lambda^{22}f.
\end{equation}
Now we state the generalized  result of Nachman \cite{Na2} as follows.
\begin{theorem}\label{t1.1}
Let $\bar{\Omega}_{1}\subset \Omega_{2}$ be a bounded Lipschitz domain in $\mathbb{R}^{n}, n\geq 2$ and let $\gamma(x)\in Lip(\Omega_{2})$ with $\gamma(x)\geq C_{0}>0$.
Then $(-\Lambda_{\gamma}-\Lambda^{11})$  is an invertible operator: $H^1({\partial{\Omega}_{1}})\rightarrow  L^{2}({\partial{\Omega}_{1}})$ and 
$ \tilde{\Lambda}_{\gamma}=\Lambda^{21}(-\Lambda_{\gamma}-\Lambda^{11})^{-1}\Lambda^{12}+\Lambda^{22}.$
\end{theorem}
Before proving Theorem \ref{t1.1}, we first introduce some known results. Let $G(x,y)$ be the Green function such that, for $x\in\Omega_{2}$, 
\begin{equation}\label{1.5}
 \left\{
\renewcommand{\arraystretch}{1.25}
\begin{array}{lll}
\divergence(\gamma\nabla G(x,y))=\delta_{x},\ \ \ \ \text{in}\ \Omega_{2},\\
 G(x,y)$\textbar$_{\partial{\Omega}_2}=0.
\end{array}
\right.
\end{equation}
Now we define the single layer potential and double layer potential as
\begin{equation}\label{1.6}
  Sf(x)=\displaystyle\int_{\partial{\Omega}_{1}} G(x,y)f(y)dS(y), \quad x\in \bar{\Omega}_{2},
\end{equation}
\begin{equation}\label{1.7}
 Df(x) =\displaystyle\int_{\partial{\Omega}_{1}} \frac{\partial{G(x,y)}}{\partial{\nu(y)}}f(y)dS(y), \quad x\in \bar{\Omega}_{2}.
\end{equation}
For the single potential $Sf(x)$ and the double layer potential $Df(x)$, we collect the following results from Mitrea and Taylor's paper \cite{MT1}
(see Proposition 1.6,  Theorem $3.7$, Proposition $3.8$ and Proposition $8.2$ of \cite{MT1}).

\begin{proposition}\label{p1.2}

If $\Omega_1$ is a bounded Lipschitz domain in $\mathbb{R}^n, n \geq 2$ and $f\in L^2(\partial{\Omega}_{1})$, then
the functions $Sf(x)$ and $Df(x)$ have the following properties:\\

(a)\ The following estimates hold
\[
 \|Df(x)\|_{L^2(\partial{\Omega}_{1})}\leq C \|f\|_{L^2(\partial{\Omega}_{1})},
\]
\[
  \|Sf(x)\|_{H^1(\partial{\Omega}_{1})}\leq C \|f\|_{L^2(\partial{\Omega}_{1})}.
\]

(b)\  $ \divergence(\gamma\nabla Sf(x))=0, x \in {\Omega}_{2}\backslash{\partial{\Omega}_{1}}$.\\

(c)\ $Sf(x)\in H^1({\Omega}_{2}-\bar{\Omega}_{1})$ and $Sf(x)\in H^1({\Omega}_{1})$. The boundary value $Sf_{+}(x)(Sf_{-}(x))$ of $Sf(x)$
from outside (respectively inside)  ${\Omega}_{1}$ are identical as elements of $H^1(\partial{\Omega}_{1})$ and agree with $Sf(x)\arrowvert_{\partial{\Omega}_{1}}$.\\

(d)\ The (nontangential) limits $\frac{\partial{Sf(x)}}{\partial{\nu}_{+}}(\frac{\partial{Sf(x)}}{\partial{\nu}_{-}})$ as the boundary is approached from 
outside (respectively inside) ${\Omega}_{1}$ are given by the formula 
\[
 \frac{\partial{Sf(x)}}{\partial{\nu}^{+}_{-}}={}^{-}_{+}\frac{1}{2}\frac{1}{\gamma(x)}f(x)+B^{*}f(x), \quad \text{ almost everywhere } x\in \partial{\Omega}_{1},
\]
where
\[
B^{*}f(x)=p.v.\displaystyle\int_{\partial{\Omega}_{1}} \frac{\partial{G(x,y)}}{\partial{\nu(x)}}f(y)dS(y).
\]
In particular 
\[
\frac{\partial{Sf(x)}}{\partial{\nu}_{-}}-\frac{\partial{Sf(x)}}{\partial{\nu}_{+}}=\frac{1}{\gamma(x)}f(x).
\]
 
(e)\ The nontangential limits $Df_{+}(x)(Df_{-}(x))$ of $Df(x)$ as we approach the boundary from outside (respectively inside)  ${\Omega}_{1}$ 
exist and satisfy 

\[
 Df_{{}^{+}_{-}}={}^{+}_{-}\frac{1}{2}\frac{1}{\gamma(x)}f(x)+Bf(x), \quad \text{ almost everywhere } x\in \partial{\Omega}_{1},
\]
where
\[
Bf(x)=p.v.\displaystyle\int_{\partial{\Omega}_{1}} \frac{\partial{G(x,y)}}{\partial{\nu(y)}}f(y)dS(y).
\]

\end{proposition}

\begin{remark}
 Our case is a little bit different from  the case of the paper \cite{MT1}. In \cite{MT1}
they dealt with the operator $ L=-\triangle+V(x) (V(x) \in L^{\infty}(M), V\geq 0, 
\text{not identical zero.})$ in a compact, connected smooth $M $ with $C^1$ metric tensor. But if we just want to get the above Proposition \ref{p1.2}, we can reproduce the steps of the proof of  Theorem $3.7$ and Proposition $3.8$ to 
show that the above Proposition \ref{p1.2} still holds for the operator $L=\divergence(\gamma\nabla u)$ in a bounded  domain 
in ${\Omega}_{2}$ with Lipchitz conductivity $\gamma$.
\end{remark}

We know that the Dirichlet-to Neumann map $\Lambda_{\gamma}$ is a bounded
 operator:$H^{\frac{1}{2}}({\partial{\Omega}_{1}})\rightarrow  H^{-\frac{1}{2}}({\partial{\Omega}_{1}})$. In fact, when ${\Omega}_{1}$ is  a  bounded Lipschitz domain 
and $\gamma(x)\in Lip ({\Omega}_{1})$,   the Dirichlet-to Neumann map $\Lambda_{\gamma}$ is a bounded operator :$H^1({\partial{\Omega}_{1}})
\rightarrow  L^{2}({\partial{\Omega}_{1}})$(the same argument 
for $\tilde{\Lambda}_{\gamma}, \Lambda^{i,j}, i, j=1, 2.$).  This result  follows from the  following Proposition \ref{p1.3}  and Proposition \ref{p1.3} can be deduced 
from Proposition $7.4$ and Proposition $8.2$ of the paper \cite{MT1}. 

\begin{proposition}\label{p1.3}
Suppose $\Omega$ is a bounded domain in $\mathbb{R}^n, n\geq 2$ with Lipchitz boundary and $\gamma(x)\in Lip ({\Omega})$ 
with $\gamma(x)\geq C_0>0$. Let $u(x)\in H^1(\Omega)$ be
 a solution to
\[
\left\{
\renewcommand{\arraystretch}{1.25}
\begin{array}{lll}
\divergence(\gamma\nabla u)=0,\ \ \ \ \text{in}\ \Omega,\\
 u$\textbar$_{\partial{\Omega}}=g\in H^1(\partial{\Omega}).
\end{array}
\right.
\] 
Then we have 
\[
 \|\frac{\partial u}{\partial \nu}\|_{L^2(\partial{\Omega})}\leq C \|g\|_{H^1(\partial{\Omega})}
\]
and 
\[
 \|(\nabla u)^{*}\|_{L^2(\partial{\Omega})}\leq C \|g\|_{H^1(\partial{\Omega})},
\]
where $(\nabla u)^{*}$ is the nontangential maximal fuction of $\nabla u$.
\end{proposition}
With the above preliminary works in hand we are now in the position to show an identity  about $\Lambda_{\gamma}+\Lambda_{11}$.

\begin{lemma}\label{l1.5}
 For any $g\in H^1(\partial{\Omega}_{1})$, we have the identity 
\[
S(\Lambda_{\gamma}+\Lambda_{11})g\arrowvert_{\partial{\Omega}_{1}}=g\arrowvert_{\partial{\Omega}_{1}}.
\]
\end{lemma}

\begin{proof}
For any $g\in H^1(\partial{\Omega}_{1})$ we can find a solution $w\in H^1({\Omega}_{2}\backslash\bar{\Omega}_{1})$ satisfying 
\[
\left\{
\renewcommand{\arraystretch}{1.25}
\begin{array}{lll}
\divergence(\gamma\nabla w)=0,\ \ \ \ \text{in}\ \  {\Omega}_{2}\backslash\bar{\Omega}_{1} ,\\
 w$\textbar$_{\partial{\Omega}_{1}}=g, \quad  w$\textbar$_{\partial{\Omega}_{2}}=0.
\end{array}
\right.
\] 

Since $\gamma(x)\in Lip ({\Omega_{2}})$ the local regularity of the uniformly elliptic equation implies that $w\in H^2_{loc}({\Omega}_{2}\backslash\bar{\Omega}_{1})$
and  Proposition \ref{p1.3} implies that $(\nabla w)^{*}\in L^2(\partial({\Omega}_{2}\backslash\bar{\Omega}_{1}))$. Therefore we can still use  Green's formula  for 
 $x\in {\Omega}_{2}\backslash\bar{\Omega}_{1}$, 
\begin{equation}\label{1.8}
\renewcommand{\arraystretch}{1.25}
\begin{array}{lll}
 w(x)&=\displaystyle\int_{{\Omega}_{2}\backslash\bar{\Omega}_{1}}\divergence(\gamma\nabla w) G(x, y)dy-
\displaystyle\int_{{\Omega}_{2}\backslash\bar{\Omega}_{1}} w \divergence(\gamma\nabla G(x, y))  dy\\
&=\displaystyle\int_{\partial{\Omega}_{2}}G(x, y)\gamma\frac{\partial w}{\partial{\nu}}dS(y)-
\displaystyle\int_{\partial{\Omega}_{2}}\frac{\partial G(x, y)}{\partial{\nu(y)}}\gamma w dS(y)\\
&+\displaystyle\int_{\partial{\Omega}_{1}}G(x, y)\Lambda^{11} g dS(y)+
\displaystyle\int_{\partial{\Omega}_{1}}\frac{\partial G(x, y)}{\partial{\nu(y)}}\gamma w dS(y).
\end{array}
\end{equation}
Since $G(x,y)=0, y\in \partial{\Omega}_{2}$ and $w\arrowvert_{\partial{\Omega}_{2}}=0$, \eqref{1.8} implies
\begin{equation}\label{1.9}
  w(x)=\displaystyle\int_{\partial{\Omega}_{1}}G(x, y)\Lambda^{11} g dS(y)+
\displaystyle\int_{\partial{\Omega}_{1}}\frac{\partial G(x, y)}{\partial{\nu(y)}}\gamma w dS(y).
\end{equation}
Let $x\rightarrow \partial{\Omega}_{1}$ in \eqref{1.9} and Proposition \ref{p1.2} implies
\begin{equation}\label{1.10}
  g(x)=\displaystyle\int_{\partial{\Omega}_{1}}G(x, y)\Lambda^{11} g dS(y)+ \frac{1}{2} g(x)+
\displaystyle\int_{\partial{\Omega}_{1}}\frac{\partial G(x, y)}{\partial{\nu(y)}}\gamma g dS(y), a.e., x\in \partial{\Omega}_{1}.
\end{equation}
On the other hand we can find a solution $u$ to
\[
\left\{
\renewcommand{\arraystretch}{1.25}
\begin{array}{lll}
\divergence(\gamma\nabla u)=0,\ \ \ \ \text{in}\ \ {\Omega}_{1} ,\\
 u$\textbar$_{\partial{\Omega}_{1}}=g.
\end{array}
\right.
\] 
Using Green's formula again we know for  $x\in \Omega_{1}$

\begin{equation}\label{1.11}
\renewcommand{\arraystretch}{1.25}
\begin{array}{lll}
 u(x)&=\displaystyle\int_{{\Omega}_{1}}\divergence(\gamma\nabla u) G(x, y)dy-
\displaystyle\int_{{\Omega}_{1}} u \divergence(\gamma\nabla G(x, y))  dy\\
&=\displaystyle\int_{\partial{\Omega}_{1}}G(x, y)\Lambda_{\gamma} g dS(y)-
\displaystyle\int_{\partial{\Omega}_{1}}\frac{\partial G(x, y)}{\partial{\nu(y)}}\gamma u dS(y).
\end{array}
\end{equation}
Let $x\rightarrow \partial{\Omega}_{1}$ in \eqref{1.11} and Proposition \ref{p1.2}  gives 

\begin{equation}\label{1.12}
g(x)=\displaystyle\int_{\partial{\Omega}_{1}}G(x, y) \Lambda_{\gamma} g dS(y)+ \frac{1}{2} g(x)-
\displaystyle\int_{\partial{\Omega}_{1}}\frac{\partial G(x, y)}{\partial{\nu(y)}}\gamma g dS(y), a.e., x\in \partial{\Omega}_{1}.
 \end{equation}
Finally \eqref{1.10} $+$ \eqref{1.12} imply
\[
 g(x)=\displaystyle\int_{\partial{\Omega}_{1}}G(x, y) (\Lambda^{11}+\Lambda_{\gamma}) g dS(y),
\]
which means 
\[
S(\Lambda_{\gamma}+\Lambda_{11})g\arrowvert_{\partial{\Omega}_{1}}=g\arrowvert_{\partial{\Omega}_{1}}.
\]
At last we need to point out that the above proof is just a formal deduction due to the singularity of $\divergence(\gamma\nabla G(x, y))$. To overcome this 
obstacle we can deal with the integrations in a domain in which we remove a small ball $B_\varepsilon(x)$ centered at the singular point $x$. Finally letting $\varepsilon\rightarrow 0$ we get the above identity. 
\end{proof}
Now we will show that $\Lambda_{\gamma}+\Lambda_{11}$ is an invertible operator: $H^1({\partial{\Omega}_{1}})\rightarrow L^2({\partial{\Omega}_{1}}) $. From this Theorem \ref{t1.1} follows.

\begin{lemma}\label{1.6}
  $\Lambda_{\gamma}+\Lambda_{11}$ is an invertible operator: $H^1({\partial{\Omega}_{1}})\rightarrow L^2({\partial{\Omega}_{1}}) $ and the invertible operator 
is $Sf\arrowvert_{\partial{\Omega}_{1}}$ for any $f\in L^2(\partial{\Omega}_{1})$.
\end{lemma}
\begin{proof}
We first show that $S\arrowvert_{\partial{\Omega}_{1}}$ is an isomorphism from  $L^2({\partial{\Omega}_{1}})$ to $H^1({\partial{\Omega}_{1}})$. 
In fact Lemma \ref{l1.5} implies  $S\arrowvert_{\partial{\Omega}_{1}}$ is surjective. To show that $S\arrowvert_{\partial{\Omega}_{1}}$  is injective we assume that $\int_{\partial{\Omega}_{1}}G(x,y)f(y)dS(y)=0$ on $\partial{\Omega}_{1}$ for $f\in L^2(\partial{\Omega}_{1})$. By Proposition \ref{1.2} we know $Sf(x)$ is a solution to $\divergence(\gamma\nabla u)=0 $ in ${\Omega}_{1}$ and $Sf_{-}(x)=0$ on the boundary $\partial{\Omega}_{1}$. Hence the maximum principle 
of the uniformly elliptic equation implies $Sf(x)\equiv 0 $ in $\bar{\Omega}_{1}$. On the other hand $Sf(x)$ is also a solution to $\divergence(\gamma\nabla u)=0$
in $\Omega_{2}\backslash \bar{\Omega}_{1}$. In view of the choice $G(x, y)=0, x\in \partial{\Omega}_{2}$, we know $Sf(x)\arrowvert_{\partial{\Omega}_{2}}=0$ and  
Proposition \ref{1.2} implies $Sf_{+}(x)=0$ on the boundary $\partial{\Omega}_{1}$, so the maximum principle of the uniformly 
elliptic equation also implies $Sf(x)\equiv 0 $ in $\bar\Omega_{2}\backslash{\Omega}_{1}$.

From the analysis above we know that $Sf(x)\equiv 0$ in $\Omega_{2}$. By Proposition \ref{1.2}, we have 
\[
\frac{1}{\gamma(x)}f(x)=\frac{\partial{Sf(x)}}{\partial{\nu}_{-}}-\frac{\partial{Sf(x)}}{\partial{\nu}_{+}}=0,
\]
 which implies $f\equiv 0$ on the boundary $\partial{\Omega}_{1}$. Thus $S\arrowvert_{\partial{\Omega}_{1}}$ is bijective and Proposition \ref{1.2} implies 
$S\arrowvert_{\partial{\Omega}_{1}}$ is a bounded operator:$L^2(\partial{\Omega}_{1})\rightarrow H^1(\partial{\Omega}_{1})$. Finally Banach inverse mapping theorem 
gives us $S\arrowvert_{\partial{\Omega}_{1}}$ is an isomorphism which $L^2({\partial{\Omega}_{1}})$ to $H^1({\partial{\Omega}_{1}})$.

On the other hand, from Lemma \ref{l1.5}'s identity $S\arrowvert_{\partial{\Omega}_{1}}(\Lambda_{\gamma}+\Lambda_{11})=I$, bijection of
 $S\arrowvert_{\partial{\Omega}_{1}}$  implies $\Lambda_{\gamma}+\Lambda_{11}$ is also bijective and from  Proposition \ref{p1.3} we have that $\Lambda_{\gamma}+\Lambda_{11}$ is a bounded operator: $H^1({\partial{\Omega}_{1}})\rightarrow L^2({\partial{\Omega}_{1}})$. Banach inverse mapping theorem 
implies $(\Lambda_{\gamma}+\Lambda_{11})^{-1}=S\arrowvert_{\partial{\Omega}_{1}}$.
\end{proof}

\section{Construction of CGO Solutions from the Boundary Integral Equation}
 In this section we will follow Nachman's idea to construct Complex Geometrical Optic solutions to the equation $\divergence(\gamma\nabla u)=0$ from the boundary integral equation. We begin with some basic properties of Green's function. For $\rho\in \mathbb{C}^n$ with $\rho\cdot\rho=0$, $(-\xi^2+2i\rho\cdot\xi)^{-1}$ is easily checked to be locally 
integrable as a function of $\xi$ in $\mathbb{R}^n, n\geq3$. Its inverse Fourier transform formally written as 
\[
g_{\rho}(x)=g(x, \rho)=\frac{1}{(2\pi)^{n}}\displaystyle\int \dfrac{e^{ix\cdot\xi}}{-\xi^2+2i\rho\cdot\xi}d\xi
\]
is a tempered distribution satisfying 
\[
 (\triangle_{x}+2\rho \cdot \nabla_{x})g(x,\rho) =\delta_{0}(x).
\]
The distribution $G_{\rho}(x)=-e^{x\cdot \rho}g(x,\rho)$ is then a fundamental solution for the Laplace operator in $\mathbb{R}^n$:
\[
 -\triangle G_{\rho}(x)=\delta_{0}(x).
\]
It follows that $G_{\rho}(x)$ differs from the Green's function of classical potential theory by a global harmonic function in $\mathbb{R}^n$:
\[
 G_{\rho}(x)=G_{0}(x)+H_{\rho}(x), \quad \triangle H_{\rho}=0 \ \text{in}\  \mathbb{R}^n,
\]
where
\[
 G_{0}(x)=g(x,0)=\frac{1}{(n-2)\omega_n}|x|^{2-n}(\omega_n=\dfrac{(2\pi)^{n/2}}{\Gamma(n/2)}).
\]
Note that $G_{\rho}(x)$ is a smooth function for $x$ away from the origin and has the same singularity near $x=0$ as that of $G_{0}(x)$.

Using the family $G_{\rho}$ of Green's function defined above we now consider analogues of the classical single and double potentials. Let 
\[
S_{\rho}f(x)=\displaystyle\int_{\partial{\Omega}} G_{\rho}(x,y)f(y)dS(y),
\]
\[
  Df_{\rho}(x) =\displaystyle\int_{\partial{\Omega}} \frac{\partial{G_{\rho}(x,y)}}{\partial{\nu(y)}}f(y)dS(y),
\]
where $G_{\rho}(x, y)=:G_{\rho}(x-y)$ and to begin with, $f$ is continuous on $\partial{\Omega}$ and $x\in\mathbb{R}^n\backslash\partial{\Omega}$.
Define the boundary, or trace , double layer potential by 
\[
 B_{\rho}f(x)=p.v.\displaystyle\int_{\partial{\Omega}} \frac{\partial{G_{\rho}(x,y)}}{\partial{\nu(y)}}f(y)dS(y), \text{for}\  x\in \partial{\Omega}.
\]
Now we collect some properties about single and double layer potentials from Nachman's paper \cite{Na1} as follows (see Lemma $2.4$ and Lemma $2.5$ of \cite{Na1}). 

\begin{proposition}\label{p3.1}
 Let $\Omega$ be a bounded domain in $\mathbb{R}^n, n\geq 3$ with $C^{1,1}$ boundary and suppose $f\in H^{\frac{1}{2}}(\partial{\Omega})$. Then the function $u=S_{\rho}f$ have the following properties:
 
(a)\ $\triangle u  =0$ in $\mathbb{R}^n\backslash\partial{\Omega}$.

(b)\ $u$ is in $H^2(\Omega)$ and in $H^2(\Omega_{\theta}^{'})$ for $\theta>\theta_{0}$, where $\theta_{0}$ is a large enough number so that $\bar{\Omega}\subset \{x:\ 
|x|<\theta_{0}\}$, and for any $\theta>\theta_{0}$, $\Omega_{\theta}^{'}=\{x:\ x\notin \bar{\Omega}, |x|<\theta\}$.

(c)The boundary values $u_{+}(u_{-})$ of $u$ from outside (respectively inside) $\Omega$ are identical as elements of $H^{\frac{3}{2}}(\partial{\Omega})$
and agree with the trace of single layer potential $S_{\rho}f$.

(d)\  $u$ satisfies the following decay properties:
\begin{equation}\label{3.1}
\lim\limits_{\theta\rightarrow \infty}\displaystyle\int_{|y|=\theta}G_{\rho}(x, y)\frac{\partial {u}}{\partial{\nu}(y)}
-u(y)\frac{\partial {G_{\rho}(x, y)}}{\partial{\nu}(y)}dS(y)=0, 
a.e., x\in \mathbb{R}^n.
\end{equation}

(e)\ Let $g\in H^{\frac{3}{2}}(\partial{\Omega})$ and $v=D_{\rho}g$ defined in $\mathbb{R}^n\backslash\partial{\Omega}$. Then the nontangential limits $v_{+}(v_{-})$
of $v$ as we approach the boundary from outside (respectively inside) $\Omega$ exist and satisfy 
\[
 v_{{}^{+}_{-}}={}^{+}_{-}\frac{1}{2}g(x)+B_{\rho}f(x)\  a.e.,x\in \partial{\Omega}.
\]
\end{proposition}
With the above preliminary works in hand we first deduce a boundary integral equation for the CGO solutions. 

From the analysis we performed in Section $2$ we can assume $\Omega=B_R(0)$
and $\gamma(x)\equiv 1$ near the boundary $\partial{B}_{R}(0)$ and for $\gamma \in Lip(B_{R}(0))(\text{or}\  \gamma \in  C^1(\bar{B}_{R}(0)))$ define the Dirichlet-to-Neumann map as follows 
\[
\Lambda_{\gamma}:\ f\in H^{\frac{3}{2}}(\partial{B}_{R}(0))\rightarrow \frac{\partial{u_{f}}}{\partial{\nu}}\in H^{\frac{1}{2}}(\partial{B}_{R}(0)),
\]
where $u_{f}$ is  the solution to 
\[
\left\{
\renewcommand{\arraystretch}{1.25}
\begin{array}{lll}
\divergence(\gamma\nabla u)=0,\ \ \ \ \text{in}\ B_{R}(0),\\
 u$\textbar$_{\partial{B}_R(0)}=f.
\end{array}
\right.
\]
Clearly $\Lambda_{\gamma}$ is a bounded operator: $H^{\frac{3}{2}}(\partial{B}_{R}(0))\rightarrow H^{\frac{1}{2}}(\partial{B}_{R}(0))$.

We are now ready to establish a one to one correspondence between the solutions of the boundary integral equation and those of the following exterior problem:
\begin{equation}\label{3.2}
\renewcommand{\arraystretch}{1.25}
\begin{array}{lll}
&(i)\ \triangle\psi=0 \ \ \ \ \ \  \ \ \ \ \ \ \ \ \ \ \ \ \ \ \ \ \ \ \ \ \ \ \ \ \ \text{in}\  \mathbb{R}^n\backslash \bar{B}_{R}(0),\\
&(ii)\ \psi\in H^2(B_{R^{'}}\backslash\bar{B}_{R}(0)), \ \ \ \ \ \ \ \ \ \ \ \ \text{for} \ R^{'}>R,\\
&(iii)\  \psi(x, \rho)-e^{x\cdot \rho} \ \ \ \ \ \ \  \ \ \ \ \ \ \ \ \ \ \ \ \ \ \ \text{satisfies}\  \eqref{3.1},\\
&(iv)\ \frac{\partial{\psi}}{\partial{\nu_{+}}}=\Lambda_{\gamma}\psi \quad \quad \quad \quad \quad \ \ \ \ \ \ \ \ \ \ \text{on}\  \partial{B}_R(0).
\end{array}
\end{equation}
We state Lemma $2.7$ of Nachman's paper \cite{Na1} as follows.

\begin{proposition}\label{p3.2}
 (a)\ Suppose $\psi$ solves $(\eqref{3.2}(i)-(iv))$. Then its trace on the boundary $f_{\rho}=\psi_{+}=\psi\arrowvert_{\partial{B}_R(0))}$ solves
\begin{equation}\label{3.3}
f_{\rho}=e^{x\cdot \rho}-(S_{\rho}\Lambda_{\gamma}-B_{\rho} -\frac{1}{2}I)f_{\rho}.
 \end{equation}
(b)\ Conversely, suppose $f_{\rho}\in H^{\frac{3}{2}}(\partial{B}_R(0)))$ solves \eqref{3.3}. Then the function $\psi(x, \rho)$ defined for
 $x$ in $\mathbb{R}^n\backslash \bar{B}_{R}(0)$ by 
\begin{equation}\label{3.4}
\psi(x, \rho)=e^{x\cdot \rho}-(S_{\rho}\Lambda_{\gamma}-D_{\rho})f_{\rho}
 \end{equation}
solves the exterior problem \eqref{3.2}$(i), (ii), (iii), (iv)$.  Furthermore $\psi\arrowvert_{\partial{B}_R(0))}=f_{\rho}$.
\end{proposition}

In the following we will use the boundary integral equation \eqref{3.3} to construct CGO solutions to $\divergence(\gamma\nabla u)=0$ in $\mathbb{R}^n$.

\begin{lemma}\label{l3.3}
 Suppose $f_{\rho}\in H^{\frac{3}{2}}(\partial{B}_R(0)))$ solves \eqref{3.3}. Then

(a)\ There exists a unique solution $u\in H^{2}_{loc}(\mathbb{R}^n)$ to 
\begin{equation}\label{3.5}
 \divergence(\gamma\nabla u)=0
\end{equation}
in $\mathbb{R}^n$ such that $u=\psi(x, \rho)$ in $\mathbb{R}^n\backslash {B}_{R}(0)$.

(b)\ Let $v=\gamma^{\frac{1}{2}}u$ and $v\in H^1_{loc}(\mathbb{R}^n)$ is a weak solution to the following Schr\"{o}dinger equation
\begin{equation*}
 -\triangle v+qv=0, \text{in} \ \mathbb{R}^n,
\end{equation*}
where $q=\dfrac{\triangle\gamma^{1/2}}{\gamma^{1/2}}$. Furthermore the following identity holds
\begin{equation}\label{3.6}
v(x)=e^{x\cdot \rho}+\displaystyle\int_{\mathbb{R}^n}\nabla {\gamma^{1/2}} \cdot \nabla(\frac{1}{\gamma^{1/2}}G_{\rho}(x, y)v) dy,\ a.e., x\in \mathbb{R}^n.
 \end{equation}

(c)\ Let $v(x)=e^{x\cdot \rho}(1+\Phi(x, \rho))$. Then $\Phi(x, \rho)\in H^1_{loc}(\mathbb{R}^n)$ satisfies
\begin{equation}\label{3.7}
\Phi(x, \rho)=\displaystyle\int_{\mathbb{R}^n} g_{\rho}(x, y)q(y)dy +\displaystyle\int_{\mathbb{R}^n} g_{\rho}(x, y)q(y)\Phi(y, \rho)dy
\end{equation}
as a tempered distribution.
\end{lemma}

\begin{proof}
For $f_{\rho}\in H^{\frac{3}{2}}(\partial{B}_R(0)))$ we can find a solution $w\in H^2(B_{R}(0))$ such that 
\[
\left\{
\renewcommand{\arraystretch}{1.25}
\begin{array}{lll}
\divergence(\gamma\nabla w)=0\ \ \ \ \text{in}\ B_{R}(0),\\
 w$\textbar$_{\partial{B}_R(0)}=f_{\rho}.
\end{array}
\right.
\]
Proposition \ref{p3.2} implies $w\arrowvert_{\partial{B}_R(0)}=\psi(x, \rho)\arrowvert_{\partial{B}_R(0)}=f_{\rho}$ and
 $\frac{\partial{w}}{\partial{\nu}_{-}}=\frac{\partial{\psi(x, \rho)}}{\partial{\nu}_{+}}$. Hence defining 
\[
u=
\left\{
\renewcommand{\arraystretch}{1.25}
\begin{array}{lll}
w \quad \quad \ \ \ \ \ \ \  \text{in}\  B_{R}(0),\\
\psi(x, \rho) \quad  \quad \text{in} \ \mathbb{R}^n\backslash  B_{R}(0),
\end{array}
\right.
\]
we know $u\in H^2_{loc}(\mathbb{R}^n)$ is a solution to 
\[
 \divergence(\gamma\nabla u)=0\ \ \ \ \text{in}\  \mathbb{R}^n.
\]
The uniqueness follows from the maximum principle of the uniformly elliptic equation. Then $(a)$ follows.

Letting $v=\gamma^{\frac{1}{2}}u$ in view of $\gamma\equiv 1$ in $\mathbb{R}^n\backslash B_{R_{0}}(0) (R_{0}<R)$, we can deduce from
 \eqref{3.5} that $v\in H^1_{loc}(\mathbb{R}^n)\bigcap H^2_{loc}(\mathbb{R}^n\backslash \bar{B}_{R_{0}}(0))$ is a weak solution to 
\begin{equation}\label{3.8}
 -\triangle v+qv=0, \text{in} \ \mathbb{R}^n,
\end{equation}
where $q=\dfrac{\triangle\gamma^{1/2}}{\gamma^{1/2}}$.

For any  $R^{'}>R$ and $\delta>0$ small enough, we can consider a Lipschitz function $h_{\delta}(t)$ given by
\[
h_{\delta}(t)=
\left\{
\renewcommand{\arraystretch}{1.25}
\begin{array}{lll}
1\ \ \ \ \ \ \  t\in [0, R^{'}]\backslash [R-\delta, R+\delta],\\
\frac{R-t}{\delta}\ \ \ t\in [R-\delta, R],\\
\frac{t-R}{\delta}\ \ \ t\in[R, R+\delta],
\end{array}
\right.
\]
and then define the function $H_{\delta}(x)=h_{\delta}(|x|)\in Lip(B_{R^{'}}(0)), x \in \bar {B}_{R^{'}}(0)$. Clearly, we have \\
\[
\renewcommand{\arraystretch}{1.25}
\begin{array}{lll}
&(i)\ |H_{\delta}(x)|\leq 1\ and \   H_{\delta}(x)\arrowvert_{\partial{B}_R(0)}=0,\\
&(ii)\ spt{|\nabla H_{\delta}(x)|}\subset T_{\delta}=:\bar{B}_{R+\delta}(0)\backslash
 B_{R-\delta}(0).
\end{array}
\]
Given $\varepsilon>0$ we define the following regularized version of $G_{\rho}(x, y)$ by 
\begin{equation}\label{3.9}
G_{\rho}^{\varepsilon}(x, y)=\frac{1}{(n-2)\omega_{n}}(|x-y|^2+\varepsilon^2)^{\frac{2-n}{2}}+H_{\rho}(x-y).
\end{equation}
Taking the test function $H_{\delta}(y)G_{\rho}^{\varepsilon}(x, y)\in H^1_{0}(B_{R}(0))$ in \eqref{3.8}, we have 
\begin{equation}\label{3.10}
\displaystyle\int_{B_{R}(0)} \nabla v\cdot \nabla (H_{\delta}(y)G_{\rho}^{\varepsilon}(x, y))dy=\displaystyle\int_{B_{R}(0)}\nabla 
{\gamma^{1/2}} \cdot \nabla(\frac{1}{\gamma^{1/2}}v H_{\delta}(y) G_{\rho}^{\varepsilon}(x, y)) dy.
\end{equation}

On the other hand since $v\in H^2({B}_{R^{'}}(0)\backslash \bar{B}_{R}(0))$ and $\gamma\equiv 1\  \text{in} \ {B}_{R^{'}}(0)\backslash \bar{B}_{R}(0),$ 
Green's formula gives us 
\begin{equation}\label{3.11}
\renewcommand{\arraystretch}{1.25}
\begin{array}{lll}
&\displaystyle\int_{{T_{R^{'}}}}\nabla 
{\gamma^{1/2}} \cdot \nabla(\frac{1}{\gamma^{1/2}}v H_{\delta}(y) G_{\rho}^{\varepsilon}(x, y))dy=
\displaystyle\int_{T_{R^{'}}}-\triangle v H_{\delta}(y)G_{\rho}^{\varepsilon}(x, y)dy\\
=
&\displaystyle\int_{T_{R^{'}}}\nabla v\cdot \nabla (H_{\delta}(y)G_{\rho}^{\varepsilon}(x, y))dy-
\displaystyle\int_{\partial{B_{R^{'}}}(0)}\frac{\partial{v}}{\partial{\nu}}G_{\rho}^{\varepsilon}(x, y)dS(y),
\end{array}
\end{equation}
where $T_{R^{'}}=B_{R^{'}}(0)\backslash \bar{B}_{R}(0).$

Now from \eqref{3.10}$+$\eqref{3.11} we have
\begin{equation}\label{3.12}
 \renewcommand{\arraystretch}{1.25}
\begin{array}{lll}
&\displaystyle\int_{B_{R^{'}}(0)}\nabla 
{\gamma^{1/2}} \cdot \nabla(\frac{1}{\gamma^{1/2}}v H_{\delta}(y) G_{\rho}^{\varepsilon}(x, y)) dy\\
=&\displaystyle\int_{B_{R^{'}}(0)}\nabla v\cdot \nabla (H_{\delta}(y)G_{\rho}^{\varepsilon}(x, y))dy-
\displaystyle\int_{\partial{B_{R^{'}}}(0)}\frac{\partial{v}}{\partial{\nu}}G_{\rho}^{\varepsilon}(x, y)dS(y)=:I+II.
\end{array}
\end{equation}
Note that the left hand side of \eqref{3.12} can be rewritten as follows
\begin{equation}\label{3.13}
 \renewcommand{\arraystretch}{1.25}
\begin{array}{lll}
&\displaystyle\int_{B_{R^{'}}(0)}\nabla 
{\gamma^{1/2}} \cdot \nabla(\frac{1}{\gamma^{1/2}}v  G_{\rho}^{\varepsilon}(x, y)) H_{\delta}(y)dy\\
&+\displaystyle\int_{B_{R^{'}}(0)}\nabla 
{\gamma^{1/2}} \cdot \nabla {H_{\delta}(y)}\frac{1}{\gamma^{1/2}}v  G_{\rho}^{\varepsilon}(x, y) dy=:III+IV.
\end{array}
\end{equation}
Letting $\delta\rightarrow0$ Lebesgue dominated convergence theorem implies 
\begin{equation}\label{3.14}
 \lim\limits_{\delta\rightarrow 0}III=\displaystyle\int_{B_{R^{'}}(0)}\nabla 
{\gamma^{1/2}} \cdot \nabla(\frac{1}{\gamma^{1/2}}v  G_{\rho}^{\varepsilon}(x, y))dy.
\end{equation}

On the other hand in view of $spt{|\nabla H_{\delta}(x)|}\subset T_{\delta}=:\bar{B}_{R+\delta}(0)\backslash
 B_{R-\delta}(0)$ and $\gamma\equiv1$ in $\mathbb{R}^n \backslash\bar{B}_{R_0}(0)$, when $\delta<R-R_{0}$ we have 
\begin{equation}\label{3.15}
IV=0.
 \end{equation}
Combining \eqref{3.12}, \eqref{3.13}, \eqref{3.14} and \eqref{3.15}, we deduce that
\begin{equation}\label{3.16}
  \lim\limits_{\delta\rightarrow 0}\displaystyle\int_{B_{R^{'}}(0)}\nabla 
{\gamma^{1/2}} \cdot \nabla(\frac{1}{\gamma^{1/2}}v H_{\delta}(y) G_{\rho}^{\varepsilon}(x, y)) dy=\displaystyle\int_{B_{R^{'}}(0)}\nabla 
{\gamma^{1/2}} \cdot \nabla(\frac{1}{\gamma^{1/2}}v  G_{\rho}^{\varepsilon}(x, y))dy.
\end{equation}
For $I$ we have 
\begin{equation}\label{3.17}
 I=\displaystyle\int_{B_{R^{'}}(0)}\nabla v\cdot \nabla G_{\rho}^{\varepsilon}(x, y)H_{\delta}(y)dy+
\displaystyle\int_{B_{R^{'}}(0)}\nabla v\cdot \nabla H_{\delta}(y)G_{\rho}^{\varepsilon}(x, y)dy=:V+VI.
\end{equation}
For $V$ by letting $\delta\rightarrow0$ Lebesgue dominated convergence theorem implies 
\begin{equation}\label{3.18}
 \lim\limits_{\delta\rightarrow 0}V=\displaystyle\int_{B_{R^{'}}(0)}\nabla v\cdot \nabla G_{\rho}^{\varepsilon}(x, y)dy.
\end{equation}
In view of  $spt|\nabla H_{\delta}|\subset T_{\delta}$ and $v\in H_{loc}^2(\mathbb{R}^n\backslash \bar{B}_{R_0})$ integration by parts 
gives us 
\begin{equation}\label{3.19}
 \renewcommand{\arraystretch}{1.25}
\begin{array}{lll}
VI&=\displaystyle\int_{T_{\delta}}\nabla v\cdot \nabla H_{\delta}(y)G_{\rho}^{\varepsilon}(x, y)dy\\
   &= -\displaystyle\int_{T_{\delta}} \triangle v H_{\delta}(y)G_{\rho}^{\varepsilon}(x, y)dy-
\displaystyle\int_{T_{\delta}}\nabla v\cdot \nabla G_{\rho}^{\varepsilon}(x, y)H_{\delta}(y)dy\\
&+\displaystyle\int_{\partial{B}_{R+\delta}(0)}\frac{\partial{v}}{\partial{\nu}}G_{\rho}^{\varepsilon}(x, y)dS(y)-
\displaystyle\int_{\partial{B}_{R-\delta}(0)}\frac{\partial{v}}{\partial{\nu}}G_{\rho}^{\varepsilon}(x, y)dS(y)
\end{array}
\end{equation}
Since $v\in H_{loc}^2(\mathbb{R}^n\backslash \bar{B}_{R_0})\bigcap C^{\infty}(\mathbb{R}^n\backslash \bar{B}_{R_0})$ and $G_{\rho}^{\varepsilon}(x, y)$
is smooth, letting $\delta\rightarrow 0$ in \eqref{3.19} we have
\begin{equation}\label{3.20}
 \lim\limits_{\delta\rightarrow 0}VI=0.
\end{equation}
Combining \eqref{3.12}, \eqref{3.17}, \eqref{3.18} and \eqref{3.20}, we have 
\begin{equation}\label{3.21}
  \lim\limits_{\delta\rightarrow 0} \displaystyle\int_{B_{R^{'}}(0)}\nabla v\cdot \nabla (H_{\delta}(y)G_{\rho}^{\varepsilon}(x, y))dy=
\displaystyle\int_{B_{R^{'}}(0)}\nabla v\cdot \nabla G_{\rho}^{\varepsilon}(x, y)dy,
\end{equation}
and from \eqref{3.12}, \eqref{3.16} and \eqref{3.21} we derive

\begin{equation}\label{3.22}
 \renewcommand{\arraystretch}{1.25}
\begin{array}{lll}
&\displaystyle\int_{B_{R^{'}}(0)}\nabla 
{\gamma^{1/2}} \cdot \nabla(\frac{1}{\gamma^{1/2}}v  G_{\rho}^{\varepsilon}(x, y)) dy\\
=&\displaystyle\int_{B_{R^{'}}(0)}\nabla v\cdot \nabla G_{\rho}^{\varepsilon}(x, y)dy-
\displaystyle\int_{\partial{B_{R^{'}}}(0)}\frac{\partial{v}}{\partial{\nu}}G_{\rho}^{\varepsilon}(x, y)dS(y).
\end{array}
\end{equation}
On the other hand, Green's formula gives us
\begin{equation}\label{3.23}
 \renewcommand{\arraystretch}{1.25}
\begin{array}{lll}
&-\displaystyle\int_{B_{R^{'}}(0)}v\triangle G_{\rho}^{\varepsilon}(x, y)
 dy\\
=&\displaystyle\int_{B_{R^{'}}(0)}\nabla v\cdot \nabla {G_{\rho}^{\varepsilon}(x, y)} dy-
\displaystyle\int_{\partial{B_{R^{'}}}(0)}\frac{\partial{G_{\rho}^{\varepsilon}(x, y)}}{\partial{\nu(y)}}vdS(y).
\end{array}
\end{equation}
From \eqref{3.22} and \eqref{3.23}, we deduce 
\begin{equation}\label{3.24}
 \renewcommand{\arraystretch}{1.25}
\begin{array}{lll}
-\displaystyle\int_{B_{R^{'}}(0)}v\triangle G_{\rho}^{\varepsilon}(x, y)&=
\displaystyle\int_{\partial{B_{R^{'}}}(0)}\frac{\partial{v}}{\partial{\nu}}G_{\rho}^{\varepsilon}(x, y)dS(y)-
\displaystyle\int_{\partial{B_{R^{'}}}(0)}\frac{\partial{G_{\rho}^{\varepsilon}(x, y)}}{\partial{\nu(y)}}vdS(y)\\
&+\displaystyle\int_{B_{R^{'}}(0)}\nabla 
{\gamma^{1/2}} \cdot \nabla(\frac{1}{\gamma^{1/2}}v  G_{\rho}^{\varepsilon}(x, y)) dy.
\end{array}
\end{equation}
Due to $\gamma\equiv 1 \ \text{in}\  \mathbb{R}^n \backslash B_{R^{'}}(0)$ and $v=\gamma^{1/2}u$, letting $\varepsilon\rightarrow 0$ 
in \eqref{3.24} we have 
for almost every $x\in B_{R^{'}}(0)$
\begin{equation}\label{3.25}
 \renewcommand{\arraystretch}{1.25}
\begin{array}{lll}
v(x)&=
\displaystyle\int_{\partial{B_{R^{'}}}(0)}\frac{\partial{u}}{\partial{\nu}}G_{\rho}(x, y)dS(y)-
\displaystyle\int_{\partial{B_{R^{'}}}(0)}\frac{\partial{G_{\rho}(x, y)}}{\partial{\nu(y)}}udS(y)\\
&+\displaystyle\int_{\mathbb{R}^n}\nabla 
{\gamma^{1/2}} \cdot \nabla(\frac{1}{\gamma^{1/2}}v  G_{\rho}(x, y)) dy,
\end{array}
\end{equation}
where we have used that $-\triangle G_{\rho}^{\varepsilon}(x, y)$ is an approximation of identity and $L^p$ estimate of the convolution type integration about the kernels
of $G_0(x-y)$ and $\nabla G_0(x-y)$.

In view of $-\triangle_{y}G_{\rho}(x, y)=\delta_{x}$ and $\triangle e^{x\cdot\rho}=0$ whenever $\ \rho\cdot\rho=0$, we deduce from Green's formula 
\begin{equation}\label{3.26}
e^{x\cdot\rho}=\displaystyle\int_{\partial{B_{R^{'}}}(0)}\frac{\partial{e^{x\cdot\rho}}}{\partial{\nu}}G_{\rho}(x, y)dS(y)-
\displaystyle\int_{\partial{B_{R^{'}}}(0)}\frac{\partial{G_{\rho}(x, y)}}{\partial{\nu(y)}}e^{x\cdot\rho}dS(y).
\end{equation}
Since $u=\psi(x, \rho)$ in $\mathbb{R}^n\backslash {B}_{R}(0)$ and Proposition \ref{3.2} implies that $u-e^{x\cdot\rho}$ satisfies \eqref{3.1}, we can deduce from
\eqref{3.26}
\begin{equation}\label{3.27}
 \lim\limits_{R^{'} \rightarrow \infty}\biggl(\displaystyle\int_{\partial{B_{R^{'}}}(0)}\frac{\partial{u}}{\partial{\nu}}G_{\rho}(x, y)dS(y)-
\displaystyle\int_{\partial{B_{R^{'}}}(0)}\frac{\partial{G_{\rho}(x, y)}}{\partial{\nu(y)}}udS(y)\biggr)=e^{x\cdot\rho}.
\end{equation}
Then \eqref{3.25} and \eqref{3.27} imply
\begin{equation}\label{3.28}
v(x)=e^{x\cdot \rho}+\displaystyle\int_{\mathbb{R}^n}\nabla {\gamma^{1/2}} \cdot \nabla(\frac{1}{\gamma^{1/2}}G_{\rho}(x, y)v) dy,\ a.e., x\in \mathbb{R}^n.
 \end{equation}
Then $(b)$ follows.
 
Let $\gamma_{\varepsilon}=J_{\varepsilon}\ast\gamma$, where $J_{\varepsilon}$ is the standard mollifier. Since $\gamma\in Lip (\mathbb{R}^n)$ and $\gamma\equiv1$ in 
$\mathbb{R}^n\backslash B_{R}(0)$, by integration by parts we can deduce from \eqref{3.28}
\begin{equation}\label{3.29}
 v(x)=e^{x\cdot \rho}-\lim\limits_{\varepsilon\rightarrow 0}\displaystyle\int_{\mathbb{R}^n}\frac{\triangle \gamma^{1/2}_{\varepsilon}}{\gamma^{1/2}_{\varepsilon}}
G_{\rho}(x, y)vdy.
\end{equation}
Recalling $v(x)=e^{x\cdot \rho}(1+\Phi(x, \rho))$ we can deduce from \eqref{3.29}
\begin{equation}\label{3.30}
 \Phi(x, \rho)=\lim\limits_{\varepsilon\rightarrow 0}\displaystyle\int_{\mathbb{R}^n} g_{\rho}(x, y)\frac{\triangle \gamma^{1/2}_{\varepsilon}}{\gamma^{1/2}_{\varepsilon}}dy
 +\lim\limits_{\varepsilon\rightarrow 0}\displaystyle\int_{\mathbb{R}^n} g_{\rho}(x, y)
\frac{\triangle \gamma^{1/2}_{\varepsilon}}{\gamma^{1/2}_{\varepsilon}}\Phi(y, \rho)dy,
\end{equation}
where we have used the relations $G_{\rho}(x, y)=G_{\rho}(x-y)$, $g_{\rho}(x, y)=g_{\rho}(x-y)$ and $g_{\rho}(x)=-e^{x\cdot \rho}G_{\rho}(x)$.

Hence, we conclude from \eqref{3.30} that
\begin{equation}\label{3.31}
\Phi(x, \rho)=\displaystyle\int_{\mathbb{R}^n} g_{\rho}(x, y)\frac{\triangle \gamma^{1/2}}{\gamma^{1/2}}dy +\displaystyle\int_{\mathbb{R}^n} g_{\rho}(x, y)\frac{\triangle \gamma^{1/2}}{\gamma^{1/2}}
\Phi(y, \rho)dy
\end{equation}
as a tempered distribution.

In fact, 
\begin{equation}\label{3.32}
 \biggl(\displaystyle\int_{\mathbb{R}^n} g_{\rho}(x, y)q(y)dy\biggr)^{\hat{}}(\xi) = \frac{1}{-|\xi|^2+2i\rho \cdot \xi}
\biggl(\frac{\triangle \gamma^{1/2}}{\gamma^{1/2}}\biggr)^{\hat{}}(\xi),
\end{equation}
and 
\begin{equation}\label{3.33}
  \biggl(\displaystyle\int_{\mathbb{R}^n} g_{\rho}(x, y)q(y)\Phi(y, \rho)dy\biggr)^{\hat{}}(\xi) = \frac{1}{-|\xi|^2+2i\rho \cdot \xi}
\biggl(\frac{\triangle \gamma^{1/2}}{\gamma^{1/2}}\Phi(y, \rho)\biggr)^{\hat{}}(\xi).
\end{equation}

For $\gamma\in Lip (\mathbb{R}^n)$ and $\gamma\equiv1$ in 
$\mathbb{R}^n\backslash B_{R}(0)$ and $\Phi(x, \rho)\in H^1_{loc}(\mathbb{R}^n)$ from the results of the paper of \cite{HT1}, we know \eqref{3.31}, \eqref{3.32} and 
\eqref{3.33} all make sense. Then $(c)$ follows.
\end{proof}
In Lemma \ref{l3.3}, we proved the identity \eqref{3.7}. We mentioned that $\int_{\mathbb{R}^n} g_{\rho}(x, y)q(y)dy  $ and $\int_{\mathbb{R}^n} g_{\rho}(x, y)q(y)\Phi(y, \rho)dy $ are tempered distributions.  In fact in the paper \cite{HT1} by Habermann and Tataru, they showed 
$\int_{\mathbb{R}^n} g_{\rho}(x, y)q(y)dy \in \dot{X}_{\rho}^{\frac{1}{2}}$ and $\int_{\mathbb{R}^n} g_{\rho}(x, y)q(y)$ $\Phi(y, \rho)dy\in \dot{X}_{\rho}^{\frac{1}{2}}$.
We are now in the position to introduce the definition of Bourgain's space $\dot{X}_{\rho}^{b}$ (see \cite{Bou}). Following the idea of Haberman and Tataru we define the space $\dot{X}_{\rho}^{b}$ by the norm 
\[
 \|u\|_{\dot{X}_{\rho}^{b}}=\||p_{\rho}(\xi)|^b\hat{u}(\xi)\|_{L^2},
\]
where $p_{\rho}(\xi)=-|\xi|^2+2i\rho \cdot \xi$ is the symbol of $\triangle_{\rho}:=\triangle+2\rho\cdot \nabla$.

Now given $k\in \mathbb{R}^n$, consider $P$ a 2-dimensional linear subpace orthogonal to $k$ and set 
\begin{equation}\label{ro1}
 \rho_{1}=s\eta_{1} +i \left(\frac{k}{2}+r\eta_{2}\right),
\end{equation}
\begin{equation}\label{ro2}
 \rho_{2}=-s\eta_{1} +i \left(\frac{k}{2}-r\eta_{2}\right),
\end{equation}
where $\eta_{1}, \eta_{2}\in S^{n-1}$ satisfy $(k, \eta_{1})=(k, \eta_{2})=(\eta_{1}, \eta_{2})=0$ and $\frac{|k|^2}{4}+r^{2}=s^{2}$. We have that $\eta_1$ can be chosen to be $\eta_1\in P\cap S^{n-1}$ (for later references set $ S := P \cap S^{n-1} $) and $ \eta_2 $ is the unique vector making $ \{ \eta_1, \eta_2 \} $ a positively oriented orthonormal basis of $ P $. The vectors are chosen so that $\rho_{i}\cdot \rho_{i}=0, i=1, 2.$ and $\rho_{1}+\rho_{2}=ik$.

We can prove the following result.
\begin{theorem}\label{t3.4}
 Let $\gamma(x)\in Lip(B_{R}(0))$ be a real valued function and assume that $\gamma(x)\geq C_0>0$ with $\gamma(x)\equiv 1$ near the boundary $\partial{B}_{R}(0)$
. Then there exists a constant $\varepsilon_{n, R}$ such that  if $\gamma$ satisfies either $\|\nabla log\gamma\|\leq \varepsilon_{n, R}$ or 
$\gamma \in C^{1}(\bar{B}_{R}(0))$, then the following properties hold

(a)\ $S_{\rho}\Lambda_{\gamma}-B_{\rho} -\frac{1}{2}I$ is a compact operator on $H^{\frac{3}{2}}(\partial{B}_{R}(0))$.

(b)\ For any $\rho\cdot\rho=0, \rho \in \mathbb{C}^{n}$, when $|\rho|$ is large enough, then there exists a unique $f_{\rho}\in H^{\frac{3}{2}}(\partial{B}_{R}(0))$ such 
that 
\[
 f_{\rho}=e^{x\cdot \rho}-(S_{\rho}\Lambda_{\gamma}-B_{\rho} -\frac{1}{2}I)f_{\rho}.
\]

(c)\ Let us consider $\rho_i=\rho_i(s,\eta_1)$ as above. We have that $u=\gamma^{-1/2}
e^{x\cdot \rho_i}(1+\Phi(x, \rho_i))$ are solutions to 
\begin{equation*}
\divergence(\gamma\nabla u)=0, \text{in} \ \mathbb{R}^n.
\end{equation*}
Moreover, for a fixed $k\in\mathbb{R}^n$ and $P$ as above we have that for sufficiently large $\lambda$,
\[
 \frac{1}{\lambda}\int_S\int_{\lambda}^{2\lambda}\|\Phi(x, \rho_i)\|_{\dot{X}_{\rho_i}^{\frac{1}{2}}} ds d\eta_1\to 0,
\]
\[
 \frac{1}{\lambda}\int_S\int_{\lambda}^{2\lambda} \|\eta_{B}q\|_{\dot{X}_{\rho_i}^{-\frac{1}{2}}} ds d\eta_1\to 0,
\]
where $\eta_{B}$ is a smooth function with compact support. 

\end{theorem}

\begin{proof}
 For $(a)$ let $g\in H^{\frac{3}{2}}(\partial{B}_{R}(0))$ and consider $u\in H^{2}(B_{R}(0))$ solution to 
\[
\left\{
\renewcommand{\arraystretch}{1.25}
\begin{array}{lll}
\divergence(\gamma\nabla u)=0,\ \ \ \ \text{in}\ B_{R}(0),\\
 u$\textbar$_{\partial{B}_R(0)}=g.
\end{array}
\right.
\]
Let $x\in \Omega$ and apply Green's formula to $G_{\rho}^{\varepsilon}(x, y)$ in $\Omega$. After taking the limit as $\varepsilon$ tends to zero we obtain
\[
\renewcommand{\arraystretch}{1.25}
\begin{array}{lll}
\displaystyle\int_{B_{R}(0)} G_{\rho} (x, y) \triangle u dy + u(x) &= \displaystyle\int_{\partial{B}_{R}(0)} G_{\rho} (x, y)\frac{\partial u}{\partial\nu}dS(y)
-\displaystyle\int_{\partial{B}_{R}(0)} G_{\rho} (x, y)\frac{\partial G_{\rho} (x, y)}{\partial\nu} g dS(y)\\
&=(S_{\rho}\Lambda_{\gamma}-D_{\rho})g(x).
\end{array}
\]
Letting $x\rightarrow \partial{B}_{R}(0) $ (nontangential from inside $B_{R}(0)$ ), we deduce from Proposition \ref{p3.1}
\begin{equation}\label{3.34}
(S_{\rho}\Lambda_{\gamma}-B_{\rho} -\frac{1}{2}I)g = T \displaystyle\int_{B_{R}(0)} G_{\rho} (x, y)\nabla log \gamma \cdot \nabla u dy, 
\end{equation}
where $T$ denotes the trace operator.

Since $\gamma \in Lip(B_{R}(0))$ from the estimates of the uniformly elliptic equation we know that the operator
 $ P_{\gamma}: g \in H^{\frac{3}{2}}(\partial{B}_{R}(0))\rightarrow u\in H^{2}(B_{R}(0))$ is bounded. Furthermore Rellich's compact embedding theorem 
implies that the operator $\Xi: u\in H^{2}(B_{R}(0))\rightarrow  \nabla log \gamma \cdot \nabla u$ is compact and Calder\'{o}n-Zygmund estimates imply that the operator
$G_{\rho}: f\in L^2(B_{R}(0)\rightarrow G_{\rho}f\in H^{2}(B_{R}(0))$ is bounded. On the other hand, the trace operator $T: H^{2}(B_{R}(0))\rightarrow 
H^{\frac{3}{2}}(\partial{B}_{R}(0))$ is bounded. Summing up  the above analysis we can deduce from \eqref{3.34} that
$S_{\rho}\Lambda_{\gamma}-B_{\rho} -\frac{1}{2}I=TG_{\rho}\Xi P_{\gamma}$
is a compact operator on $ H^{\frac{3}{2}}(\partial{B}_{R}(0))$. Then $(a)$ follows.

To prove $(b)$ by Fredholm alternative theorem we just need to show that the homogeneous equation 
\begin{equation}\label{3.35}
 f_{\rho}=(-S_{\rho}\Lambda_{\gamma}+B_{\rho} +\frac{1}{2}I) f_{\rho}
\end{equation}
only has the zero solution. For any $g\in H^{\frac{3}{2}}(\partial{B}_{R}(0))$ satisfying \eqref{3.35} repeating the steps of the proof of Lemma \ref{l3.3} we can find
a solution $u(x, \rho)=\gamma^{-1/2}
e^{x\cdot \rho}\tilde{\Phi}(x, \rho) \in H^2_{loc}(\mathbb{R}^n)$ to 
\[
\renewcommand{\arraystretch}{1.25}
\begin{array}{lll}
\divergence(\gamma\nabla u)=0,\ \ \ \ \text{in}\ \mathbb{R}^n,
\end{array}
\]
with $u$\textbar$_{\partial{B}_R(0)}=g$ and 
\begin{equation}\label{3.36}
 \tilde{\Phi}(x, \rho)=\displaystyle\int_{\mathbb{R}^n} g_{\rho}(x, y)q(y)\tilde{\Phi}(y, \rho)dy.
\end{equation}
Noting $\tilde{\Phi}(x, \rho) \in H^1_{loc}(\mathbb{R}^n)$ and $q=\frac{\triangle \gamma^{1/2}}{\gamma^{1/2}}$ with compact
support, by dual method we can show 
\begin{equation}\label{3.37}
q(y)\tilde{\Phi}(y, \rho)\in \dot{X}_{\rho}^{-\frac{1}{2}}.
 \end{equation}
Then \eqref{3.36} and \eqref{3.37} imply $\tilde{\Phi}(x, \rho) \in \dot{X}_{\rho}^{\frac{1}{2}} $.

Under our assumptions: $\gamma\in Lip(B_{R}(0))$ and $\|\nabla log\gamma\|\leq \varepsilon_{n, R}$ or $\gamma\in C^1(\bar{B}_{R}(0))$,  when
$\rho$ is large enough Haberman and Tataru \cite{HT1} showed the operator :
\[
 f(x)\in \dot{X}_{\rho}^{\frac{1}{2}} \longrightarrow \displaystyle\int_{\mathbb{R}^n} g_{\rho}(x, y)q(y)f(y)dy \in \dot{X}_{\rho}^{\frac{1}{2}}
\]
is a contraction. Therefore it follows from \eqref{3.36} that $\tilde{\Phi}(x, \rho)\equiv 0$. Hence $u\equiv 0$ and $g\equiv 0$. Then $(b)$ follows.

In view of $(b)$, for any $\rho\cdot\rho=0, \rho \in \mathbb{C}^{n}$, when $|\rho|$ is large enough, then there exists a unique $f_{\rho}\in H^{\frac{3}{2}}(\partial{B}_{R}(0))$ such 
that 
\[
 f_{\rho}=e^{x\cdot \rho}-(S_{\rho}\Lambda_{\gamma}-B_{\rho} -\frac{1}{2}I)f_{\rho}.
\]
Now Lemma \ref{l3.3} implies that $u(x)=\gamma^{-1/2}
e^{x\cdot\rho}(1+\Phi(x, \rho))$ is a  solution to 
\begin{equation*}
\divergence(\gamma\nabla u)=0, \ \text{in} \ \mathbb{R}^n,
\end{equation*}
and 
\begin{equation} \label{3.38}
\Phi(x, \rho)=\displaystyle\int_{\mathbb{R}^n} g_{\rho}(x, y)q(y)dy +\displaystyle\int_{\mathbb{R}^n} g_{\rho}(x, y)q(y)\Phi(y, \rho)dy.
\end{equation} 
Finally from \eqref{3.38} under our assumptions about $\gamma$, Lemma $3.1$ and Theorem $4.1$ in \cite{HT1} imply that $(c)$ follows.
\end{proof}

\section{Reconstruction of the Conductivity $\gamma$}

\begin{theorem}\label{t4.1}  
 Under the assumptions of Theorem \ref{t3.4} $\gamma$ can be recovered from the knowledge of $\Lambda_{\gamma}$ at the boundary $\partial{B}_{R}(0)$.
\end{theorem}

Before proving Theorem \ref{t4.1} we first recall an integral identity appearing in Lemma $4.1$ in \cite{Knu}.

\begin{proposition}\label{p4.2}
Suppose $\gamma_{i}\in C^1(\bar{\Omega}), i=1, 2.$ and $u_1, u_2\in H^1(\Omega)$
satisfy $\nabla\cdot(\gamma_i\nabla u_i)=0$ in $\Omega$. Suppose further that $\tilde{ u}_1\in H^1(\Omega)$
satisfies $\nabla\cdot(\gamma_1\nabla \tilde{ u}_1)=0$ with $\tilde{ u}_1=u_2$ on $\partial \Omega$. Then
 \[
 \begin{array}{lll}
\displaystyle\int_{\Omega}(\gamma_1^{\frac{1}{2}}\nabla \gamma_2^{\frac{1}{2}}-\gamma_2^{\frac{1}{2}}\nabla \gamma_1^{\frac{1}{2}})\cdot \nabla (u_1 u_2)dx=
\displaystyle\int_{\partial \Omega}\gamma_1 \partial_{\nu}(\tilde{ u}_1-u_2)u_1ds,
\end{array}
\]
where the integral on the boundary is understood in the sense of the dual pairing between $H^{\frac{1}{2}}(\partial \Omega)$ and $H^{-\frac{1}{2}}(\partial \Omega)$.
\end{proposition}

\begin{remark}
 Proposition \ref{p4.2} also holds for $\gamma_i\in Lip({\Omega})$.
\end{remark}

\textbf{Proof of Theorem \ref{t4.1}}

\begin{proof}
Taking  the conductivities $\gamma_1=\gamma$ and $\gamma_2=1$, $\rho_1$ and $\rho_2$ as in \eqref{ro1} and \eqref{ro2} respectively and the solutions $u_1=\gamma^{-1/2}
e^{x\cdot \rho_{1}}(1+\Phi(x, \rho_{1}))$ and $u_2=e^{x\cdot \rho_{2}}$ in Proposition \ref{p4.2}, we have 
\begin{equation}\label{4.1}
 \displaystyle\int_{B_{R}(0)}-\nabla\gamma^{1/2}\cdot \nabla(\gamma^{-1/2}e^{x\cdot \rho_{1}}e^{x\cdot \rho_{2}}(1+\Phi(x, \rho_{1}))dx
=\displaystyle\int_{\partial{B}_{R}(0)}(\Lambda_{\gamma}e^{x\cdot \rho_{2}}-\gamma\frac{\partial{(e^{x\cdot \rho_{2}})}}{\partial{\nu}})u_{1}dS.
\end{equation}

We have that $u_{1}\arrowvert_{\mathbb{R}^n\setminus\bar{B}_{R}(0)}$ is a solution of the exterior problem \eqref{3.2} such that $u_{1}\arrowvert_{\partial{B}_{R}(0)}=f_{\rho_{1}}$. Since $e^{x\cdot \rho_{1}}\arrowvert_{\partial{B}_{R}(0)}$, $S_{\rho_{1}}$, $B_{\rho_{1}}$ are known and $\Lambda_{\gamma}$ and $\gamma$ at $\partial B_{R}(0)$ can be determined, Proposition \ref{p3.2} and $(b)$ of Theorem \ref{t3.4} imply that the right hand side of \eqref{4.1}  is known.  
Since we have that $\rho_{1}+\rho_{2}=ik$ we deduce from \eqref{4.1}
\begin{equation}\label{4.2}
 \displaystyle\int_{B_{R}(0)}-\nabla\gamma^{1/2}\cdot \nabla(\gamma^{-1/2}e^{ix\cdot k}(1+\Phi(x, \rho_{1})))dx
=\displaystyle\int_{\partial{B}_{R}(0)}(\Lambda_{\gamma}e^{x\cdot \rho_{2}}-\gamma\frac{\partial{(e^{x\cdot \rho_{2}})}}{\partial{\nu}})f_{\rho_{1}}dS.
\end{equation}
Let $\gamma_{\varepsilon}$ be as in Lemma \ref{l3.3}. Now integration by parts in the left hand side of \eqref{4.2} gives us
\begin{equation}\label{4.3}
 \lim\limits_{\varepsilon\rightarrow 0}\displaystyle\int_{B_{R}(0)}\frac{\triangle \gamma_{\varepsilon}^{1/2}}{\gamma_{\varepsilon}^{1/2}}e^{ix\cdot k}(1+\Phi(x, \rho_{1}))dx
=\displaystyle\int_{\partial{B}_{R}(0)}(\Lambda_{\gamma}e^{x\cdot \rho_{2}}-\gamma\frac{\partial{(e^{x\cdot \rho_{2}})}}{\partial{\nu}})f_{\rho_{1}}dS.
\end{equation}
Clearly we have 
\begin{equation}\label{4.4}
 \lim\limits_{\varepsilon\rightarrow 0}\displaystyle\int_{B_{R}(0)}\frac{\triangle \gamma_{\varepsilon}^{1/2}}{\gamma_{\varepsilon}^{1/2}}e^{ix\cdot k} 
= \displaystyle\int_{B_{R}(0)}\frac{\triangle \gamma^{1/2}}{\gamma^{1/2}}e^{ix\cdot k}=\hat{q}(k)   
\end{equation}
in the sense of the tempered distribution. 

Since $\|\gamma_{\varepsilon}-\gamma\|_{H^1{(\mathbb{R}^n)}}\rightarrow 0$ as $\varepsilon\rightarrow 0$ it holds
\begin{equation}\label{4.5}
 \lim\limits_{\varepsilon\rightarrow 0}\displaystyle\int_{B_{R}(0)}
\frac{\triangle \gamma_{\varepsilon}^{1/2}}{\gamma_{\varepsilon}^{1/2}}e^{ix\cdot k}\Phi(x, \rho_{1})dx=\langle\frac{\triangle \gamma^{1/2}}{\gamma^{1/2}}
e^{ix\cdot k}\varphi_{B},\  \Phi(x, \rho_{1})\rangle
\end{equation}
where $\langle  \quad \rangle$ means the dual pairing between $\dot{X}_{\rho_{1}}^{\frac{1}{2}}$ and  $\dot{X}_{\rho_{1}}^{-\frac{1}{2}}$ and 
$\varphi_{B}\in C_{0}^{\infty}(\mathbb{R}^n)$ with $\varphi_{B}\equiv1$ in $B_{R}(0)$.
Therefore, we have from \eqref{4.3}, \eqref{4.4} and \eqref{4.5} that the following holds
\begin{equation}\label{fourier}
\hat{q}(k)+\langle\frac{\triangle \gamma^{1/2}}{\gamma^{1/2}}
e^{ix\cdot k}\varphi_{B},\  \Phi(x, \rho_{1})\rangle=\displaystyle\int_{\partial{B}_{R}(0)}(\Lambda_{\gamma}e^{x\cdot \rho_{2}}-\gamma\frac{\partial{(e^{x\cdot \rho_{2}})}}{\partial{\nu}})f_{\rho_{1}}dS.
\end{equation}
For a fixed $k\in\mathbb{R}^n$ and $\lambda\geq |k|$, we next take average of the last identity in $(s,\eta_1)\in [\lambda, 2\lambda]\times S$. We get that
\begin{align}\label{average}
\frac{1}{\lambda}\int_S\int_{\lambda}^{2\lambda}\hat{q}(k)ds d\eta_1+\frac{1}{\lambda}\int_S\int_{\lambda}^{2\lambda}&\langle\frac{\triangle \gamma^{1/2}}{\gamma^{1/2}}
e^{ix\cdot k}\varphi_{B},\  \Phi(x, \rho_{1})\rangle ds d\eta_1\\
&=\frac{1}{\lambda}\int_S\int_{\lambda}^{2\lambda}\int_{\partial{B}_{R}(0)}(\Lambda_{\gamma}e^{x\cdot \rho_{2}}-\gamma\frac{\partial{(e^{x\cdot \rho_{2}})}}{\partial{\nu}})f_{\rho_{1}}dS ds d\eta_1
\nonumber
\end{align}
By (c) of Theorem \ref{t3.4} in Section 3 we have that 
\begin{equation}
\frac{1}{\lambda}\int_S\int_{\lambda}^{2\lambda}\langle\frac{\triangle \gamma^{1/2}}{\gamma^{1/2}}
e^{ix\cdot k}\varphi_{B},\  \Phi(x, \rho_{1})\rangle ds d\eta_1\to 0,
\end{equation}
as $\lambda\to\infty$. Since the right hand side of \eqref{average} is still known and $\hat{q}(k)$ does not depend neither on $s$ nor on $\eta_1$, we have that  $\hat{q}(k)$ is known as a tempered distribution. By inverting the Fourier transform the potential $q(x)$
can be recovered in $\mathbb{R}^n$. From the definition of $q(x)=\frac{\triangle \gamma^{1/2}}{\gamma^{1/2}}$ we know $w=log\sqrt{\gamma}\in H_{0}^{1}(B_{R}(0))$ is a weak 
solution to
\begin{equation}\label{4.7}
\left\{
\renewcommand{\arraystretch}{1.25}
\begin{array}{lll}
 \triangle w+|\nabla w|^2=q, \ \ \ \ \text{in}\ B_{R}(0), \\
 w\arrowvert_{\partial{B}_{R}(0)}=0.
\end{array}
\right.
\end{equation}
The nonlinear Dirichlet problem \eqref{4.7} has a unique solution by the maximum principle of uniform elliptic equations and since we have already recovered $q(x)$ we may construct $\gamma$ in $B_{R}(0)$ by solving the equation \eqref{4.7}.
In other words $\gamma$ can be recovered by the knowledge of $\Lambda_{\gamma}$ at the boundary $\partial{B}_{R}(0)$. Then Theorem \ref{t4.1} follows.

\end{proof}

\section[Appendix by R. M. Brown, A. Garc\'ia and G. Zhang]{Appendix.}

\begin{center}

 \large{Recovering the gradient of a $C^1$-conductivity at the boundary}\\

\small{${}^{c}$R. M. Brown\footnote{
E-mail address:\ ${}^{c}$Russell.Brown@uky.edu\\
This work was partially supported by a grant from the Simons
Foundation (\#195075 to Russell Brown). } \  Andoni Garc\'ia  \   Guo Zhang\\
${}^{c}$Department of Mathematics,  University of Kentucky \\
Lexington, Kentucky 40506, USA}

\end{center}

The goal of this appendix is to give a method for recovering the gradient
of a coefficient from the Dirichlet to Neumann map. The argument is an
extension
of a method from earlier 
work of Brown \cite{Bro1}.


Throughout this appendix, we let $\Omega$ be a Lipschitz domain as defined
in  \cite{Ve1}, for example, and we let $ \gamma $ denote a
function on $ \bar  \Omega$ that is  continuous and  satisfies  the
ellipticity condition for some $ \lambda >0$, 
\begin{equation}\label{Elliptic}
\lambda \leq \gamma \leq \lambda ^ { -1} .
\end{equation}
We will use $H^s$ to denote the standard scale of $L^2$
Sobolev spaces. 
Given $ f \in H^ {1 /2}( \partial \Omega)$, we may solve the
Dirichlet problem
$$
\left \{ \begin{array}{ll}
\divergence \gamma \nabla u = 0 \qquad & \mbox{in } \Omega\\
u=f , \qquad & \mbox{on } \partial \Omega
\end{array}
\right.
$$
We  define  $ \gamma \partial u / \partial \nu$,  the co-normal
derivative of $u$,  as an element of $H^ {-1/2}( \partial \Omega)$
by
$$
\langle \gamma \frac { \partial u }{ \partial \nu } , \phi \rangle 
= \int _ \Omega \gamma \nabla u \cdot \nabla \phi\,  dy  .
$$
Here, $ \langle \cdot , \cdot \rangle : H^ {-1/2}( \partial
\Omega) \times H^ {1/2}( \partial \Omega) \rightarrow
\mathbb{C}$ is the bilinear pairing of duality and we use $ \phi$ to
denote both a function in $ H^ {1/2 } ( \partial \Omega)$ and an
extension into 
 $\Omega$ which lies
in  $H^ {1}(
\Omega)$. Because $u$ is a solution, the right-hand side depends only
on the boundary values and not on the particular extension chosen for
$ \phi$. 

We give a reformulation of  a result of the author Brown
for recovering the conductivity at the boundary  \cite{Bro1}. 

\begin{theorem}
\label{Brown}
 Let $ \Omega$ be a Lipschitz domain. For almost every $
  x \in \partial \Omega$, we may find a sequence of functions $ f_N$
which are Lipschitz on $ \partial \Omega$, supported in $ \{ y :
|x-y | < N ^ { -1/2}\}$, satisfy 
$$ 
\| f_N \|_ {L^ { 2, s} ( \partial \Omega )} \leq C N^ { s-1/2} ,
\qquad 0 \leq s \leq 1, 
$$
and so that if $ \gamma $ is a continuous function satisfying
(\ref{Elliptic}),  and $u_N$ is   the solution of  the Dirichlet
problem for $ \divergence \gamma \nabla$  with data $f_N$, and 
 $ \psi $  is  a continuous function in $ \bar \Omega$, we have
$$
\psi(x) = \lim_ { N \rightarrow \infty } \int _ { \Omega } \psi (y)
|\nabla u_N (y) |^ 2 \, dy.
$$
\end{theorem}

As an immediate consequence of the above theorem, we obtain recovery
at the boundary and a stability result. 
The  stability result for the recovery of the boundary values of a
continuous coefficient in smooth domains was  proved by Sylvester and Uhlmann
\cite[Theorem 0.2]{SU2}.   

\begin{theorem} 
Let $ \gamma$ be continuous function in $ \bar \Omega $  and let $x
\in \partial \Omega$ and  $
f_N$ be as in Theorem \ref{Brown}, then we have
$$
\gamma (x) = \lim _{ N \rightarrow \infty } \langle \Lambda _ \gamma f
_N , \bar f_N \rangle. 
$$

As a consequence, we obtain the stability result: If $ \gamma$ and $
\tilde \gamma$ are continuous on $ \bar \Omega$ and elliptic, then 

$$
\| \gamma -  \tilde \gamma\| _ { L^ \infty ( \Omega)} \leq C \| \Lambda
_ \gamma - \Lambda _ { \tilde \gamma } \| _ { { \cal L } ( L^ { 2,
    1/2}( \partial \Omega) , L^ { 2, -1/2} ( \partial \Omega))}.
$$
\end{theorem}

\begin{proof} This follows immediately from Theorem \ref{Brown}. 
\end{proof}

If $ \gamma$ is continuously differentiable in the closure of
$\Omega$,  we have additional
regularity of the  solution  $ u_N$. Since the boundary values $ f_N$
lie in $ L^ { 2,1 } ( \partial \Omega)$, we have that  the
full gradient of $u_N$  lies in $ L^ 2 ( \partial \Omega)$ and we
obtain 
estimates for the non-tangential maximal function  of the gradient $
(\nabla u _N ) ^ *$
\begin{equation}
\label{NTEstimates}
\| ( \nabla u_N ) ^ * \|_ { L^ 2 ( \partial \Omega)}
\leq C \| f _N\| _ { L ^ { 2,1 } ( \partial \Omega)} \leq CN^ { 1/2}. 
\end{equation}
This result holds for elliptic operators in Lipschitz domains with
$C^1$ coefficients and follows from the work of Mitrea and Taylor on
manifolds with $C^1 $-metrics \cite[Section 7]{MT1}. 
We will use the estimate   (\ref{NTEstimates})  to justify the use of
the  divergence theorem for expressions involving the  gradient of a
solution   on the boundary. 
  In our next result, we let $ \nabla _t u = \nabla u
- \nu \frac {\partial u } { \partial \nu }$ denote the tangential
component of the gradient of $u$.     The following theorem gives a
recipe for recovering $ \nabla \gamma$ from the gradient of $f_N$ on
the boundary, the boundary values of $ \gamma$, 
 and the Dirichlet to Neumann  map acting on $f_N$. 

\begin{theorem}
Let $ \Omega$ be a Lipschitz domain and let $ x $ and $f_N$ be as in
Theorem \ref{Brown}. Let $ \alpha $ be a constant vector. We have
$$
\alpha \cdot \nabla \gamma (x) = 
\lim _ { N \rightarrow \infty } \int _ { \partial \Omega } 
(\gamma |\nabla _t f_N |^ 2 + \frac 1 \gamma |\Lambda _ \gamma f_N| ^
2) \alpha \cdot \nu 
-2 \Re  ( ( \Lambda _ \gamma f_N  ) \,  \alpha  \cdot \nabla \bar  u_N  ) \,
d\sigma 
$$
\end{theorem}
\begin{proof}
The rather mysterious expression on the right is obtain by rewriting
the Rellich identity \cite{Re1}  using  $\frac 1 \gamma  \Lambda _
\gamma f_N$ for 
the  normal derivative. The following identity holds because the
integrands are equal as functions
\begin{multline*}
\int _ { \partial \Omega } 
(\gamma |\nabla _t f_N |^ 2 + \frac 1 \gamma |\Lambda _ \gamma f_N  |^
2) \alpha \cdot \nu  
-2 \Re  ( \Lambda _ \gamma f_N \, \alpha  \cdot \nabla\bar  u_N  ) \,
d\sigma 
\\
= \int _ {\partial \Omega} \gamma |\nabla u _N| ^ 2 \alpha \cdot \nu - 2 \Re ( \gamma \frac  {
  \partial u }{ \partial \nu } \, \alpha\cdot \nabla \bar u_N ) \, d\sigma .
\end{multline*}
Recalling that $u_N$ is   a solution,  an application of the
divergence theorem gives 
$$
\int _ {\partial \Omega } \gamma |\nabla u _N| ^ 2 \alpha \cdot \nu - 2 \Re ( \gamma \frac  {
  \partial u }{ \partial \nu }\, \alpha\cdot \nabla \bar u_N ) \, d\sigma 
= \int_ \Omega \alpha \cdot \nabla \gamma \,  |\nabla u _N | ^2 \, dy .
$$
We use the non-tangential maximal function estimate (\ref{NTEstimates}) to justify the
application of the divergence theorem.   With this, the theorem
follows from the properties of $ u_N$ given in Theorem \ref{Brown}. 
\end{proof}

\begin{acknowledgements}
The authors want to express their gratitude to Professor Mikko Salo and to Pedro Caro for many useful suggestions and deep comments.  The authors would like to
 sincerely thank Professor R. M. Brown for pointing out the result appearing in the appendix.
The first author is supported by the projects ERC-2010 Advanced Grant, 267700 Ð InvProb and Academy of Finland
 (Decision number 250215, the Centre of Excellence in Inverse Problems) and also belongs to the project MTM2007-62186 Ministerio de Ciencia y
 Tecnolog\'ia de Espa\~na. 
 The second author is  supported by Academy of Finland 
(Project number 2100001796, the Finnish Centre of Excellence in Inverse Problems).
\end{acknowledgements}



\end{document}